%% file: paper.tex
\documentclass{article}
\usepackage{amsthm} 
\usepackage{varioref}
\usepackage{amssymb}
\usepackage{amsmath}
\usepackage{tikz}
\usepackage{hyperref}
\providecommand{\abs}[1]{\lvert#1\rvert}
\providecommand{\norm}[1]{\lVert#1\rVert}
\usepackage{color}
\usepackage[latin1]{inputenc}
\usepackage[T1]{fontenc}
\usepackage{pifont}

\begin{document}

\numberwithin{equation}{section}
\newtheorem{cor}{Corollary}[section]
\newtheorem{df}[cor]{Definition}
\newtheorem{rem}[cor]{Remark}
\newtheorem{theo}[cor]{Theorem}
\newtheorem{pr}[cor]{Proposition}
\newtheorem{lem}[cor]{Lemma}

\renewcommand{\)}{\right)}
\newcommand{\mR}{{\mathbb R}}
\newcommand{\mE}{{\mathbb E}}
\newcommand{\mF}{{\mathbb F}}
\newcommand{\mC}{{\mathbb C}}

\newcommand{\eqb}{{\begin{equation}}}
\newcommand{\eqe}{{\end{equation}}}
\renewcommand{\a}{{\alpha}}
\renewcommand{\o}{{\overline{\Omega}}}
\renewcommand{\l}{\left(}
\renewcommand{\r}{\right)}

\title{On Convergence of Solutions to Equilibria for   \\ Fully Nonlinear Parabolic Systems with \\ Nonlinear Boundary Conditions  }
\author{Helmut Abels\footnote{Fakult\"at f\"ur Mathematik,  
Universit\"at Regensburg,
93040 Regensburg,
Germany, e-mail: {\sf helmut.abels@mathematik.uni-regensburg.de}},\ \
Nasrin Arab\footnote{Fakult\"at f\"ur Mathematik,  
Universit\"at Regensburg,
93040 Regensburg,
Germany}, 
Harald Garcke\footnote{Fakult\"at f\"ur Mathematik,  
Universit\"at Regensburg,
93040 Regensburg,
Germany}
}


\date{\today}

\maketitle
\begin{abstract}
Convergence  to stationary solutions  in fully nonlinear parabolic
systems with general nonlinear boundary conditions is shown  in situations  where the set of stationary solutions  creates a  $C^2$-manifold of finite dimension which is normally stable. We apply the parabolic Hölder setting   which allows  to deal with nonlocal terms  including  highest order point evaluation. In this direction  some theorems concerning the linearized systems is also extended.
 As an application of  our  main result  we prove that the   lens-shaped networks generated by circular arcs are stable under the surface diffusion flow.
\end{abstract}

\smallskip
\noindent \textbf{Mathematics Subject Classification (2000):}

\noindent
35B35, 35K55, 35K50,  37L10, 53C44, 35B65
\input{introduc}

\input{setting}
\input{main-result}

\input{lens-shaped}

\input{append}


\medskip

\noindent 
\textbf{Acknowledgement:} The second author was supported by DFG, GRK 1692 "Curvature, Cycles, and Cohomology" during much of the work and    by Bayerisches Programm zur Realisierung der Chancengleichheit für Frauen in Forschung und Lehre und nationaler MINT-Pakt during the last year. The support is gratefully acknowledged.

\bibliographystyle{plain}
\bibliography{references-1}

\end{document}

%% file: introduc.tex
\section{Introduction}
This work is motivated by the the appearance of nonlocal, nonlinear  terms (highest order point evaluations) together with   general nonlinear boundary conditions when  studying the stability for   the fourth-order geometric flow, the surface diffusion flow, with triple junctions.

There are several questions arising  doing this studying: Which setting for  function spaces  can be used for the system of PDEs arising  from such geometric problems? Which class of nonlinear parabolic systems   can model it? 
 Having in mind that we  should  also take care of  nonlinear boundary conditions, finally do the well-known theorems about stability cover  such general problem?

Let us first look closely to the nature of  geometric problems.
In most geometric flows, the stationary solutions are
 invariant under  translation and  under dilation. (This is the case for example for the volume preserving
mean curvature flow and for the surface diffusion flow.)  Therefore typically, we are in a situation where the set of stationary solutions creates locally   a smooth finite-dimensional   manifold.  A simple approach for  proving  stability for such problems  is  the generalized principle of linearized stability.
 
Such an approach was introduced by Prüss, Simonett and Zacher \cite{Pruss20093902}  for  abstract quasilinear problems and also for vector-valued quasilinear parabolic systems with vector-valued nonlinear boundary conditions in the framework of $L_p$-optimal regularity. This approach was extended in \cite{pruss2009612} to cover a wider range of   settings and a wider range of classes of nonlinear parabolic equations, including fully nonlinear equations but just for abstract evolution equations, i.e., without nonlinear boundary conditions. 

However,  for geometric flows  with triple junctions, because of the highest-order point evaluation in the corresponding  parabolic system (due to the movement of triple junction), one cannot work in
a standard $L_p$-framework, as e.g. in \cite{Pruss20093902}. Moreover, the general nonlinear boundary conditions    (due to 
the contact, angle,   curvature and flux conditions) prevent an application of the results of Prüss et al. in \cite{pruss2009612}, which deal with abstract evolution
equations in general function spaces. 

The purpose of this paper is to extend the approaches given in \cite{Pruss20093902,pruss2009612}
to cover fully nonlinear parabolic  systems with general nonlinear boundary conditions in  parabolic Hölder spaces. Within this classical setting, i.e., the  parabolic Hölder setting we are allowed  to deal with those  nonlocal terms.
 
We  have achieved our desired objective which we summarize here:
Suppose that for a fully
nonlinear parabolic  system with general nonlinear boundary conditions we have a
finite dimensional $C^2$-\textit{manifold of equilibria} $\mathcal{E}$ such that at a point
$u_*\in \mathcal{E}$, the null space  $N(A_0)$ of the linearization $A_0$ is
given by the tangent space of $\mathcal{E}$ at  $u_*$, zero is a semi-simple eigenvalue of $A_0$, and the 
rest of the spectrum  of $A_0$ is stable.
Under these assumptions our main result states
that solutions with
initial data close to  $u_*$ exist globally in the classical sense  and converge  towards the manifold of equilibria, i.e., to some
point  on $\mathcal E$ as time tends to infinity.

In a forthcoming paper    we plan to apply our main result  to show that the stationary solutions of the form of  the standard planar double bubbles are stable under the surface diffusion flow.
It is worth noting that for the surface diffusion flow for closed hypersurfaces   Escher, Mayer and Simonett
\cite{Escher-Mayer-Simonett} used \textit{center manifold theory}
to deal with this situation. In fact they showed that the dimension of
the set of equilibria coincides with the dimension of the center
manifold which then implies that both sets have  to coincide.
This then implies stability.
Typically it is difficult to apply the theory of center manifolds
and this is in particular true for parabolic equations
involving highly nonlinear boundary conditions.
 
The paper is organised as follows. In Section \ref{setting} we formulate the problem and in Section \ref{Main Result} we state and prove our main result, i.e., Theorem \ref{theo01}.   The proof depends upon  results for the asymptotic behavior of linear  systems which are  given in the appendix. In this direction, extending the result stated in \cite{lunardisinestrariwah}, we construct  explicitly an  extension operator for  the case of vector-valued unknowns (see Subsection \ref{Sec. extension operator}). 

As an application of our main result we show in Section  \ref{application}   that   the   lens-shaped networks generated by circular arcs  are stable under the surface diffusion flow. Indeed the  lens-shaped networks are the simplest examples of the more general triple junctions where the resulting PDE has nonlocal terms in the highest order derivatives, see \eqref{mu} and \eqref{nonlinear, nonlocal}.  Therefore we work in function spaces which yield classical solutions. 

The  proof of the main theorem follows  \cite{Pruss20093902,pruss2009612}, i.e.,  it is based on  reducing the system  to its "normal form" by means of
spectral projections. However, there are differences mainly coming from  the different natures of the function spaces used: Obviously, the assumption $(A_2)$ in \cite{pruss2009612},  used to get the estimates  on functions $T$ and $R$, see \eqref{normal form} below, needed for applying  the assumption $(A_4)$ in \cite{pruss2009612},   is not satisfied in the parabolic Hölder setting. To overcome this difficulty we have derived  these estimates directly from the smoothness assumptions on the nonlinearities, see Proposition \ref{TRS} below (cf. \cite[Proposition 10]{LatushkinPrussSchnaubelt}). Moreover,  in the parabolic Hölder setting we have 
$$
\mathbb{E}_1(J)=C^{1+\frac{\alpha}{2m}}(J, X)\cap B(J_, X_{1}) \,, 
$$   
which  is  clearly not  continuously embedded in $C(J, X_1)$, i.e.,
somehow the condition $(A_1)$ in \cite{pruss2009612} is violated. As a
result we have to give      more arguments in  step (f) of our proof,
based on the existence theorem on an arbitrary  large time interval, see
Proposition \ref{existence at large} below. Furthermore, as mentioned
before, we need to show the asymptotic behavior for linear inhomogeneous
systems in   parabolic Hölder spaces   whose  counterpart is
available in the $L_p$-setting.

%% file: setting.tex
\section{Fully nonlinear parabolic systems with \\ general nonlinear boundary conditions in a \\parabolic Hölder setting}\label{setting}
Let $\Omega\subset\mR^n$ be a bounded domain of class $C^{2m + \a}$ with  boundary $\partial \Omega$, where $m \in \mathbb N$ and $0<\a<1$. Let also $\nu (x)$ denote the  outer normal of $\partial \Omega$ at $x \in \partial \Omega $. We consider the  nonlinear boundary value problem 
\begin{equation}
\left\{
 \begin{aligned} \label{eq1}
   \partial _t u(t, x) + A (u(t,\cdot))(x) &= F(u(t,.)) (x),  &  &x\in \overline{\Omega}\,,         & &t>0\,, \\[0.1cm] 
   B_j (u(t,\cdot))(x)                     &= G_j(u(t,.))(x), &  &x\in\partial\Omega\,,
        & &j= {1,\dots, mN}\,,\\[0.1cm]
   u(0, x)                                 &= u_{0}(x),       &  &x\in \overline{\Omega}\,,         \\[0.1cm]
 \end{aligned}
\right.
\end{equation}
where $u : \o \times [ 0, \infty ) \rightarrow \mathbb{R}^N$ and $A$ is a linear $2m$th-order differential operator of the form
\begin{equation*}
(Au)(x)=\sum_{|\gamma|\leq 2m}a_\gamma(x)\nabla^{\gamma}u (x)\,, \quad x\in\overline{ \Omega}\,.
\end{equation*}
Moreover, $B_j$ are  linear differential operators of order  $m_j,$
\begin{equation*}
 (B_ju)(x)=\sum_{|\beta|\leq m_{j}}b^{j}_\beta(x)\nabla^{\beta}u (x)\,,\quad x\in \partial \Omega \,,\quad j= 1,\dots,mN\,. 
\end{equation*}
Here the coefficients $a_\gamma(x)\in \mathbb{R}^{N\times N}$, $b^{_{j}}_\beta(x)\in\mathbb{R}^N$ and 
\begin{equation*}
0 \leq m_1 \leq m_2 \leq \cdots \leq m_{mN} \leq 2m-1 \,.
\end{equation*} 
Furthermore $n_j \geq 0$ denotes the number of $j$th-order boundary conditions for $j=0,\dots, 2m-1$.

We now follow \cite{BraunerHulshofLunardi,Lunardi2002385} in making the following assumptions  on the fully nonlinear terms $F$ and $G_j$ as well as on the smoothness of the coefficients:
\begin{list}{}{\leftmargin=1.0cm\topsep=0.2cm\itemsep=0.1cm\labelwidth=1cm}
\item[(H1)]
$F:B(0,R)\subset C^{2m}(\overline{\Omega})\rightarrow C(\overline{\Omega}) $ is $C^1$ with Lipschitz continuous derivative, $F(0)=0, F'(0)=0,$ and the restriction of $F$ to $B(0,R)\subset C^{2m+\a}(\overline\Omega)$ has values in $C^\a(\overline\Omega) $ and is continuously differentiable.
\item[]
$G_j:B(0,R)\subset C^{m_{j}}(\overline{\Omega})\rightarrow C(\partial\Omega) $ is $C^2$ with Lipschitz continuous  second-order derivative,
$G_j(0)=0, G_{j}'(0)=0,$ and the restriction of $G_{j}$ to $B(0,R)\subset C^{2m+\a}(\overline\Omega)$ has values in $C^{2m+\a-m_j}(\partial\Omega) $ and is continuously differentiable.
\end{list}
\begin{list}{}{\leftmargin=1.0cm\topsep=0.1cm\itemsep=0.0cm\labelwidth=1cm}
\item[(H2)\,]
The elements of the matrix $a_\gamma (x)$ belong to $C^\a(\o).$
\item[] 
The elements of the matrix $b^{j}_\beta(x)$ belong to $C^{2m+\a-m_j}(\partial \Omega).$ 
\end{list}
In assumption (H1) we have written for simplicity $C^s(K)$ instead of $C^s(K)^N$ for $K=\overline{\Omega},\partial\Omega$. In the same way all function spaces in the following will be vector-valued with a dimension that is determined by the context.

Finally, let $B=(B_1,\dots, B_{mN})$ and $G=(G_1,\dots ,G_{mN})$.
\begin{rem}\label{rem01}
Note that assumption (H1) allows for very general nonlinearities; for instance, $F$ can depend on $ D^{\alpha}u (x_{0})$, where $x_0$ is a point in $\o$ with $\abs{\a}=2m$, which is a nonlocal dependence.
\end{rem}

 As one  guesses from our assumptions above, we are interested in  classical solutions and therefore we use the following  setting:
$$
X=C( \overline\Omega), \quad
X_0=C^\a(\o),\quad
X_1=C^{2m+\a}(\o) \,.  
$$
Note that $X_1 \hookrightarrow X_0 \hookrightarrow X$. We write  $\abs{\,\cdot\,}_j$ for the norm on $X_j$ ($j=0,1$) and $\abs{\,\cdot\,}$ for the norm on $X$. Additionally, let $Y$ be a normed vector space. Then the open ball of radius $r > 0$ centered at $u \in Y$ will be denoted by $B_{Y}(u,r)$.
 
Let us now denote by $\mathcal{E}\subset B_{X_1}(0,R)$ the set of stationary solutions (equilibria) of \eqref{eq1}, i.e.,
\begin{equation}\label{set of equilibria E}
u\in \mathcal{E}\iff  u\in B_{X_1}(0,R)\,, \: Au=F(u)\quad \mbox{in } \Omega \quad \mbox{and} \quad Bu=G(u)\quad \mbox{on }\partial \Omega \,.
\end{equation}
It follows from  assumption (H1) that $u_* \equiv 0$ belongs to $ \mathcal{E}$. Although  $u_*$ is zero, we will often write $u_*$ instead of $0$ to emphasize that we deal with an equilibrium.

We follow \cite{Pruss20093902} in assuming that $u_*$ is contained in a $k$-dimensional manifold of equilibria, i.e.,  we assume that there is a neighborhood $U\subset \mR^k$ of $0\in U$, and a $C^2$-function $ \Psi:U\rightarrow X_1$, such that 
\begin{alignat}{1}
   \bullet \quad &\Psi (U)\subset \mathcal{E} \text{ and } \Psi(0)=u_*\equiv         0 , \nonumber\\
   \bullet \quad & \text{the rank of $\Psi'(0)$ equals $k$,} \nonumber\\
   \bullet \quad & A\Psi (\zeta)= F(\Psi(\zeta )) \quad \mbox{in }\Omega,
        \quad \mbox{for all }\zeta\in U, \label{def equilibria1}\\
   \bullet \quad & B\Psi (\zeta)= G(\Psi(\zeta )) \quad \mbox{on }\partial\Omega,         \quad \mbox{for all }\zeta\in U . \label{def equilibria2}
\end{alignat}

In addition we finally require that there are no other stationary solutions near $u_*$ in $X_1$ than those given by $\Psi(U)$, i.e., for some $r_1>0$,
$$\mathcal{E}\cap B_{X_1}(u_*,r_1)=\Psi(U) \, .$$
The linearization of \eqref{eq1} at $u_*$ is given by the operator $A_0$ which is the realization of A with homogeneous boundary conditions in $X=C(\o)$, i.e., the operator with domain
\begin{eqnarray}\label{linear operator}
\begin{array}{l}
D(A_0)=\Big\{ u\in C(\o)\cap\bigcap\limits_{1< p<+\infty }W^{2m,p}(\Omega):\,\ Au\in X, \quad Bu=0 \text{ on } \partial \Omega\Big\},\\[20pt]
\quad A_0u= A u,\quad u\in D(A_0)\,,
\end{array}
\end{eqnarray}
where we  used the fact that $F'(0) = G'(0) = 0$. Note that by assumption (H2), we have 
$$
 \left. A_{0}\right|_{C^{2m+\a}(\o)}: \left.C^{2m+\a}(\o)\right|_{N(B)}\rightarrow C^\a(\o) \,.
$$
\begin{rem}\label{isolated}
Since $\Omega$ is bounded, $D(A_{0})$ is compactly embedded into $C(\o)$, the resolvent operators $(\lambda I-A_{0})^{-1}$ are compact for all $\lambda \in \rho(A_0)$, and the spectrum $\sigma ( A_0 )$ consists of a sequence of isolated eigenvalues.
\end{rem}
Next we turn to  the property of the optimal regularity in  the parabolic Hölder spaces.  To this end it is just enough  to  take care of the principal parts of the linear operators $A$ and $B$, i.e.,
\begin{align*}
A_*(x, D)    &=\sum_{|\gamma|= 2m} i^{2m}a_\gamma(x)D^{\gamma} \,, \\
B_{j*}(x, D) &=\sum_{|\beta|= m_{j}}i^{m_j}b^{j}_\beta(x)D^{\beta}\,, \quad (j=1,\dots,mN)
\end{align*}
where  $D = -i\nabla $. With this notation we have $\nabla^{\beta} = i^{|\beta|} D^\beta$.  Based on the results of V.A. Solonnikov
\cite{Solonnikov19653},  the following conditions, i.e., strong parabolicity of  $A_* $\ and  the Lopatinskii-Shapiro
condition for $(A_*, B_*)$ are sufficient for H\"older-optimal regularity of $A_0$, see  Theorem VI.21 in  \cite{Eidelman1998298}: 
\begin{list}{}{\leftmargin=1.0cm\topsep=0.2cm\itemsep=0.1cm\labelwidth=1cm}
\item[(SP)]$A $ is strongly parabolic:
For all $x\in \o, \, \,\xi\in \mR^n ,| \xi |=1, \, $$$\sigma(A_*(x,\xi))\subset \mathbb{C}_+\,.$$
\item[(LS)]
(Lopatinskii-Shapiro condition) For all $x\in \partial \Omega,\ \xi \in \mR ^n$, with $\xi\cdot\nu (x)=0  , \lambda\in\overline{\mathbb{C}_+},\, \lambda\neq 0$, and $h \in \mathbb{C}^{mN},$ the system of ordinary differential equations on the half-line 
\begin{align*}
\lambda v(y)+ A_*(x, \xi+i\nu(x )\partial _y)v (y )= &0\,, && y>0,\\
 B_{j*}( x, \xi+i\nu(x )\partial _y)v (0 )=& h_{j}\,, && j= 1,\dots , mN \,,
\end{align*}
admits a unique solution $v\in C_0\l\mR_0^+;\mathbb{C}^N\r$,
\end{list}
where $C_0\l\mR_0^+;\mathbb{C}^N\r$ is the space of continuous functions which vanish at infinity.
\begin{rem}
The  strong  parabolicity condition, i.e., (SP) implies  the root condition  (cf. Amann \cite[Lemma 6.1]{Amann1985} or Morrey \cite[P. 255]{Morrey}).  Concerning the Complementing Condition (LS),  here it is formulated in a non-algebraic way but one can find the equivalence of this formulation to the algebraic formulation in Eidelman and Zhitarashu \cite[Chapter I.2]{Eidelman1998298}. See also Lemma 6.2 in \cite{Amann1985}. 
\end{rem}
We  continue by collecting  the following basic results on generation of analytic semigroups, the characterization of related interpolation spaces and elliptic regularity in H\"older spaces for the associated elliptic systems:\begin{theo}\label{known result}
Under the conditions  (H2),(SP) and (LS) the following statements   hold. \begin{description}
\item[(i)] 
The operator $-A_0$  is sectorial.
\item[(ii)]
For each $\theta \in (0,1) $ such that $2m\theta \notin \mathbb{N}$, we have \begin{align*}
D_{-A_0}(\theta, \infty)= \big \{ \varphi \in C^{2m\theta}(\o): \quad B_j \varphi =0 \text{ if } m_j \leq[2m\theta]\big \}
\end{align*} 
and the $C^{2m\theta}$-norm is equivalent to the $D_{-A_0}(\theta, \infty)$-norm.
\item[(iii)] For each $k=1, \dots , 2m-1$ we have 
\begin{align*}
C^k_\mathcal B (\o):=\{ \varphi \in C^{k}(\o): \quad B_j \varphi =0 \text{ if } m_j <k \} \hookrightarrow D_{-A_0}(\frac{k}{2m}, \infty)\,,
\end{align*}
where $C^k_\mathcal B (\o)$ is given the norm of $C^k(\o)$.
\item[(iv)] 
We have the inclusion 
\begin{align*}
\bigg \{ \varphi \in \bigcap_{p>1} W^{2m,p}(\Omega):\quad& A \varphi \in C^\a(\o), \quad B_j \varphi \in C^{2m+\a-m_j}(\partial \Omega) \,,\\& \quad j=1, \dots, mN \bigg \} \subset C^{2m+\a}(\o) 
 \end{align*}
and there exist a constant $C$ such that 
\begin{equation}\label{schauder estimate}
\| \varphi \| _{C^{2m+\a}(\o)} \leq C \bigg ( \|A \varphi \|_{C^\a(\o)}+ \| \varphi \|_{C(\o)}+\sum\limits_{j=1}^{mN} \|B_j \varphi \|_{C^{2m+\a-m_j}(\partial \Omega)}\bigg).
\end{equation}
\end{description}
\end{theo}
\begin{proof}
   The proof is an adaptation of the proof of \cite[Theorem~5.2]{lunardisinestrariwah}, where the case of a single elliptic equation is proved.   Concerning (i) and (ii), see \cite[Remark 5.1]{AcquistapaceTerreni}. (iii) follows  from the characterization of $D_{-A_0}(\frac{k}{2m}, \infty)$ provided in \cite{Acquistapace}, see precisely Remark 5.1 in \cite{Acquistapace}. In order to prove (iv) one uses that the results of \cite{AgmonDouglisNirenbergII} imply the estimate \eqref{schauder estimate}. Moreover, the inclusion in $C^{2m+\a } ( \o )$ is a consequence of the existence theorems in \cite[Section 5]{GeymonatGrisvard}. \end{proof}

Let us now differentiate \eqref {def equilibria1} and \eqref{def equilibria2} w.r.t. $\zeta$ and evaluate  them at $ \zeta = 0 $ to obtain  
\begin{equation}\label{psi}
\left\{
 \begin{aligned}
    A\Psi'(0) &=0 & &\mbox{ in }\Omega \,,\\
    B\Psi'(0) &=0 & &\mbox{ on }\partial\Omega.
 \end{aligned}
\right.
\end{equation}
We therefore see that    the range  $R ( \Psi'(0 ) )$ is contained in the null space  $ N(A_0 )$ of $A_0$. In other words 
\begin{equation} \label{inclusion N}
T_{u_*}(\mathcal{E}) \subseteq N(A_0) \,,
\end{equation}
where $T_{u_*} ( \mathcal E )$  represents the tangent space of $\mathcal E$ at the point $u_*$.

Finally we make an additional assumption on the coefficient $b^j_\beta$ known as \textquoteleft normality condition\textquoteright, which   will be used in the construction of the extension operator presented in the appendix:  
\begin{eqnarray}\label{normality condition}
\left\{
\begin{array}{l}
\text{for each } x\in \partial \Omega, \text{ the matrix } 
  \begin{pmatrix}
  \sum_{| \beta | = k}b^{j_1}_\beta(x)(\nu(x))^\beta\\ \vdots \\ \sum_{| \beta | = k}b^{j_{n_{k}}}_\beta(x)(\nu(x))^\beta
  \end{pmatrix}
  \text{is surjective},\quad\\ [0.9cm]
  \text{where } \{j_i : i=1, \dots, n_k \}=\{j : m_j = k \} \,. 
\end{array}
\right.
\end{eqnarray}
Note that $b_\beta^j(x)\in\mathbb{R}^N$ for all $x\in\partial\Omega$.
\begin{rem}

In general, the normality condition \eqref{normality condition} is not implied by the (L-S) condition, see e.g. \cite[Remark 1.1]{Acquistapace}.\end{rem}
 In the following, the compatibility conditions  read as follows.
For $j$ such that $m_j=0$ and $ x\in \partial \Omega$
\begin{equation} \label{compatibility}
\left\{
 \begin{aligned}
    B u_0                  &= G( u_0 )\,, \\
    B_j( A u _0- F( u_0 )) &= G_j' (u_0)(A u_0-F(u_0)) \,.\\
 \end{aligned}
\right.
\end{equation}

%% file: main-result.tex
\section{Main result} \label{Main Result}
This section is devoted to the statement and proof of  our main theorem on stability of stationary solutions of the nonlinear system \eqref{eq1}.\begin{theo}\label{theo01}
Let $ u_* \equiv 0 \in X_1$ be a stationary solution  of \eqref{eq1}, and assume that the regularity conditions (H1), (H2), Lopatinskii-Shapiro
condition  (LS), strong parabolicity (SP) and finally the normality condition \eqref{normality condition} are satisfied.    Moreover let $A_0$ denote the
linearization of \eqref{eq1} at  $u_*\equiv0$  defined in \eqref{linear operator}, and require that $u_*$ is normally stable, i.e., suppose that
\begin{description}
\item[(i)]
near $u_*$ the set of equilibria $ \mathcal{E}$ is a $C^{2}$-manifold  in $X_1$ of dimension $k\in \mathbb{N}$,
\item[(ii)]
the tangent space of $ \mathcal{E}$ at $u_*$ is given by $N\l A_0\r$,
\item[(iii)]
the eigenvlaue $0$ of $A_0$ is  semi-simple, i.e., $R\l A_0\r\oplus N \l A_0 \r= \MakeUppercase{X,}$
\item[(iv)]
$\sigma \l A_0\r \backslash\ \{0\}\subset \mathbb{C}_+=\{z\in \mathbb{C}: \mbox{Re }z>0\}.$
\end{description}
Then the stationary solution $u_*$ is stable in $X_1$. Moreover, if $u_0$ is sufficiently close to $u_*$ in $X_1$ and  satisfies the compatibility conditions \eqref{compatibility}, then  the unique solution $u\l t\r $ of \eqref{eq1} exists globally and approaches some  $u_\infty\in \mathcal{E}$ exponentially fast in $X_1$ as $t \to \infty$.
\end{theo}
\begin{proof}
We follow the strategy of \cite{Pruss20093902,pruss2009612}, i.e., to reduce the system \eqref{eq1} to its normal form by means of a near-identity, nonlinear transformation of variables. This in turn makes it easier to analyze the system. The proof will be done in  steps (a)-(g) and some intermediate results will be formulated as lemmas and propositions.\\[0.15cm]
\textbf{(a)} 
According to  Remark \ref{isolated}, $0 \in \sigma ( A_0 )$ is isolated  in $\sigma\l A_0\r$ which together with assumption (iv)  gives the following decomposition of   
$$
\sigma\l A_0\r=\{0\}\cup \sigma_s, \quad \sigma_s\subset \mathbb{C}_+=\{z\in \mathbb{C}: \mbox{Re }z>0\}.
$$
into two disjoint pieces.

Let  $P^{l},\, l\in \{c,s\}$, be the spectral projections associated to  $\sigma_c=\{ 0\}$ and $\sigma_{s}$, i.e., \begin{equation} \label{spectral projection}
   P^c=\frac{1}{2\pi i }\int_\gamma R(\lambda, A_0 )\,\mathrm{d} \lambda  \qquad \text{ and } \qquad P^s = I - P^c 
\end{equation} 
(see \cite[Definition A.1.1]{Lunardi1995424}).
  We set $X_j^l:=P^lX_j$ and $X^l:=P^l X$ for $l\in \{c,s\}$ and $j\in\{0,1\}$, equipped with the norms $\abs{\, \cdot\,}_j$ and  $\abs{\, \cdot\,}$ respectively for $j \in \{ 0,1 \}$.  Moreover we define the part of $A_0$ in $X^l$ by    $$
   A_l  = P^l A_{0} P^l  \quad \text{ for }  l \in \{ c, s \}\,.
   $$
\begin{lem}\label{projection}
$\left.P^c\right|_{C^\a(\o)}\in \mathcal{L}(C^{\a}(\o), C^{2m+\a}(\o)) $
\end{lem}
\begin{proof}
At first  we show $\left.R(\lambda, A_{0})\right|_{C^\a(\o)}:C^{ \a}(\o)\rightarrow C^{2m+\a}(\o)$ for $\lambda\in \rho (A_{0}) $. If we take $f\in C^\a (\o) $ and define $u:=R(\lambda , A_0)f$, then $u \in D(A_0)$ and $u$ solves 
\begin{equation*}
\left\{ 
 \begin{aligned}
    ( \lambda I - A) u &=f\in C^\a(\o) \,, \\
    B u&=0 \,.
 \end{aligned}
\right.
\end{equation*}
 By  the elliptic regularity theory precisely Theorem \ref{known result} (iv)   we get $u\in C^{2m+\a}(\o)$ and 
\begin{align*}
\| u \|_{C^{2m+\a}(\o)} \leq C(\|f\|_{C^\a(\o)} + \|u\|_{C(\o)}) \,.
\end{align*}
In other words,
\begin{align*}
\| R(\lambda, A_0) f \|_{C^{2m+\a}(\o)} \leq C(\|f\|_{C^\a(\o)} + \|R(\lambda, A_0) f \|_{C(\o)}) \,.
\end{align*}
And now by \eqref{spectral projection} and the fact that $R(\lambda, A_0) \in \mathcal{L}(X,X)$, the claim follows.
\end {proof}
Note that Lemma \ref{projection} in particular implies $\left.P^l\right|_{C^{2m+\a}(\o)}\subset C^{2m+\a}(\o)$ for $l\in \{c,s\}$.
Since $0$ is a semi-simple eigenvalue of $A_0$, we have $X^c=N(A_0)$ and $X^s=R(A_0)$ (see \cite[Proposition A.2.2]{Lunardi1995424}) and so $P^{c} $ and $P^{s}$ are the projections onto $N(A_0)$ respectively $R(A_0)$. Consequently  $A_c\equiv 0 $ which is equivalent to say $AP^c\equiv0 $ and $BP^c\equiv 0$.
Note that $N(A_0) \subset X_1$ by elliptic regularity precisely Theorem \ref{known result} (iv).

Since $X_0^c\hookrightarrow X^c\hookrightarrow  X_1 $, we get  $X_0^c=X_1^c=X^c=N(A)$. As $X^c$ is a finite dimensional vector space, all the norms are equivalent. Therefore we choose $\abs{\, \cdot\,}$ as a norm on $X^c$. Furthermore, we take as a norm on $X_j$ and $X$
\begin{equation} \label{norm on X}
\left\{
 \begin{aligned}
    \abs{u}_j &:=\abs{P^cu}+\abs{P^s u}_j & &\mbox{ for } j=0,1 \,, \\
    \abs{u}   &:=\abs{P^c u}+\abs{P^s u} \,.
 \end{aligned}
\right.
\end{equation}
\textbf{(b)} Next let us  demonstrate  that near $u_*$, the manifold $\mathcal{E}$ is   the graph of a function $\phi:B_{X^c}\l 0,\rho_0\r\rightarrow X_1^s$ . 
To this end we define the mapping  
$$
g:U\subset \mR ^k\rightarrow X^c , \quad g\l \zeta \r :=P^c\Psi\l \zeta \r, \quad\zeta \in U\,. 
$$
Taking into account the fact the  $\dim X^c = \dim \mR^k = k$, It can be easily seen by our assumptions that  $g'\l 0\r = P^c\Psi ' (0):\mR^k\rightarrow X^c$ is  bijective. Thus, we can apply the inverse function theorem to conclude   that  $g$ is a $C^2$-diffeomorphism of a neighborhood of $0$ in $\mR^k$ onto a neighborhood of $0$ in $X^c$,  which we choose as $B_{X^c}(0,\rho_0)$ for some $\rho_0>0$. 
Hence the inverse $g^{-1}:B_{X^c}(0,\rho_0)\rightarrow 
 U$ is $C^2$ and $g^{-1}(0)=0$. If we define $ \Phi(v):= \Psi (g^{-1}(v))$ for $ v\in B_{X^c}(0,\rho_0) $, we obtain $\Phi\in C^2 (B_{X^c}(0,\rho_0), X_1)$, $\Phi(0) = 0$ as well as
$$
 \{u_{*}+ \Phi(v):\, \, v\in B_{X^c}(0,\rho_0)\}= \mathcal{E}\cap W 
$$
It is easy to observe that,
$$
P^c\Phi(v)=\l\l P^c\circ \Psi \r \circ g^{-1} \r(v)= (g\circ g^{-1} )(v) = v, \quad v \in B_{X^c}(0,\rho_0),  
$$
Hence 
$$ 
\Phi(v)= P^c \Phi(v)+ P^s \Phi(v)= v+P^s \Phi(v)\quad \text{for all }v\in B_{X^c}(0,\rho_0).
$$
If we finally define $\phi(v):=P^s \Phi(v)$ and use the fact that $\Psi'(0 )(\mathbb{R}^k)\subseteq N(A_0 )$,   we obtain
\begin {equation}\label{graph}
\phi \in  C^2 (B_{X^c}(0,\rho_0), X_1^s), \quad \phi(0)=\phi'(0)=0,
\end {equation}
and 
\begin{equation}\label{equilibria}
\ \{u_{*}+ v+\phi(v):\, \, v\in B_{X^c}(0,\rho_0)\}= \mathcal{E}\cap W \,,
\end{equation}
 for some neighborhood $W$ of $u_*$ in $X_1$.

Hence we have  established our assertion, i.e., near $ u_* $ the manifold $\mathcal{E}$ can be represented as   the graph over its tangent space  $T_{u_*} \mathcal E = N ( A_0 ) = X^c$ via   the function $\phi$. Now  applying  $P^c$ and $P^s$,    equations for the equilibria of \eqref{eq1}, i.e., 
\eqref{def equilibria1} and \eqref{def equilibria2} is equivalent to the system 
\begin{eqnarray}
\begin{array}{c}
P^c A\phi(v)=P^c F(v+\phi(v) ),  \\[0.3cm]
P^s A\phi(v)=P^s F(v+\phi(v) ),\quad B\phi(v)= G(v+\phi(v)),
\end{array}
\label{n f equilibria}
\end{eqnarray}
 where
$v\in B_{X^c}(0,\rho_0)$. Here we have used the fact that $ v+\phi(v) = \Psi(g^{-1}(v)) $ for $v \in B_{X^c}(0,\rho_0)$ as well as $A_c\equiv 0$. 

 For later convenience we choose  $\rho_0 $ so small that \begin{equation}\label{phi}
\abs{\phi'(v)}_{\mathcal{L}(X^c, X_1^s)}\leq 1, \quad \abs{\phi(v)}_1\leq \abs{v}, \quad \text{ for all } v\in B_{X^c}(0,\rho_0). 
\end{equation}
For $ r \in (0, \rho_0)$, we set 
$$
\eta(r)=\text{sup}\{\norm{\phi'(\varphi)}_{\mathcal{L}(X^c, X^{s}_1)}:\varphi\in B_{X^c}(0,r)\}. 
$$
Since $\phi'(0)=0$, $\eta(r)$ tends to $0$ as $r\rightarrow 0$. Let $L' > 0$ be such that, for all $\varphi,$ $\psi\in B_{X^c}(0,r)$ with $r \in ( 0, \rho_0)$
$$
\norm{\phi'(\varphi)-\phi'(\psi)}_{\mathcal{L}(X^c,X^s_1)}\leq L'\abs{\varphi-\psi}.
$$ 

\textbf{(c)} Now we are in a position to  reduce the system \eqref{eq1} to    normal form. For this we let 
$$
v:=P^{c}u, \quad w:=P^{s}u -\phi(P^{c}u).
$$
As an immediate consequence of this change of variables we see that   
$$
\mathcal E  \cap W= B_{X^c} ( 0, \rho_0 ) \times \{ 0 \} \subset X^c \times  X_1^s \,.
$$ 

Under this nonlinear transformation of variables,  \eqref{eq1} is transformed into the following system 
\begin{equation} \label{normal form}
\left\{
 \begin{aligned}
    \partial _t v           &=T(v,w)              & &\mbox{ in }\Omega, \\
    \partial _t w+P^s AP^sw &=R(v,w)              & &\mbox{ in }\Omega, \\
    Bw                      &=S(v,w)              & &\mbox{ on }\partial\Omega,\\     v(0)                    &=v_0, \quad w(0)=w_0 & &\mbox{ in }\Omega,
 \end{aligned}
\right.
\end{equation}
with $v_0=P^cu_0$ and  $w_0=P^su_0-\phi(P^cu_0)$, where the function $T, R$ and $S$ are given by 
\begin{equation*}
\begin{array}{l}
T(v,w)=P^cF\big( v+\phi(v)+w\big)-P^c A\phi(v)-P^cAw, \\[0.1cm]
R(v,w)= P^sF\big( v+\phi(v)+w\big)-P^s A\phi(v)-\phi'(v)T(v,w), \\[0.1cm]
S(v,w)=G\big( v+\phi(v)+w\big)-B\phi(v). \\[0.1cm]
\end{array}
\end{equation*}
Now we  need to rewrite the expression for $T, R$ and $S$ into the  following more useful  form: 
\begin{equation*}
\begin{array}{l}
T(v,w)=P^c\big( F\big( v+\phi(v)+w\big)-F\big(v+\phi(v)\big)\big)-P^cAw, \\[0.1cm]
R(v,w)= P^s\big( F\big( v+\phi(v)+w\big)-F\big(v+\phi(v)\big)\big)-\phi'(v)T(v,w), \\[0.1cm]
S(v,w)=G\big( v+\phi(v)+w\big)-G\big(v+\phi(v)\big), 
\end{array}
\end{equation*}
where we have benefited from  the equilibrium equations in \eqref{n f equilibria}.

Clearly, 
$$
R(v,0)=T(v,0)=S(v,0)=0, \quad v\in B_{X^c}(0,\rho_0). 
$$

\textbf{(d)} We shall use the parabolic H\"older spaces
$$
\mathbb{E}_1(a):=C^{1+\frac{\alpha}{2m}, 2m+\alpha}(I_a\times\o)=C^{1+\frac{\alpha}{2m}}(I_a, X)\cap B(I_a, X_{1}) \, ,
$$
$$
\mathbb{E}_0(a):=C^{\frac{\alpha}{2m},\alpha}(I_a\times\o)=C^{\frac{\alpha}{2m}}(I_a, X)\cap B(I_a, X_{0}) \,.
$$
Here $0<a\le\infty$,
\begin{equation*}
I_a:=
\begin{cases}
[0,a] & \text{ for } a>0, \\
[0,\infty) & \text{ for } a = \infty \, ,
\end{cases}
\end{equation*}
and $B(I_a , X_{j})$  is a space of all bounded functions $f: I_a \rightarrow X_j$ equipped with the supremum norm. 

  Similarly we introduce the following  spaces for functions defined on  the boundary   
\begin{align*}
   \mF_j(a):= &C^{1+\frac{\alpha}{2m}-\frac{m_{j}}{2m}, 2m+\alpha-m_j}
        (I_a\times\partial\Omega) \\ 
            = &C^{1+\frac{\alpha}{2m}-\frac{m_{j}}{2m}}\l I_a, C(\partial\Omega)\r\cap B\l I_{a},C^{ 2m+\alpha-m_j}(\partial\Omega)\r, 
\end{align*}
and 
\[
\mF(a)= \prod ^{mN}_{j=1}\mF_j(a).
\]

By \eqref{norm on X} you can easily show that $\norm{p_l u}_{\mE_i(a)}\leq\norm{u}_{\mE_i(a)}$ for $i=0,1$ and $l\in \{c,s\},$  which we will use  several times without mention it. The proof of the following Lemma is given  in \cite[Theorem 2.2]{lunardisinestrariwah}. 
\begin{lem}\label{parabolic estimate}
The following continuous embedding holds with an embedding constant independent of $a,$ with $0<\theta<2m+\a$.
\begin{align*}
\mathbb{E}_1(a)\hookrightarrow C^{\frac{\theta}{2m}}(I_a, C^{2m+\a-\theta}(\o))\,.
\end{align*}

\end{lem}
We now state the optimal regularity theorem for  the linear system 

\begin{equation} \label{linear}
\left\{
 \begin{aligned}
    \partial _t u+Au &= f(t)  & &\mbox{in }\Omega,         & &t \in (0,a)         \,, \\ 
     Bu              &= g(t)  & &\mbox{on }\partial\Omega, & &t \in (0,a)         \,,\\
     u(0)            &= u_{0} & &\mbox{in }\Omega \,, \\
 \end{aligned}
\right.
\end{equation}
 in the parabolic H\"older setting
. See  Theorem VI.21 in \cite{Eidelman1998298}.\\ In the following we need the compatibility conditions 
\begin{equation} \label{compatibility linear}
\left\{
 \begin{aligned}
        B u_0                &= g(0) \,,\\
        B_j f(0) - B_j A u_0 &= \partial_t  g_j (t) |_{t=0}\quad \text{ for all } j \text{ such that } m_j = 0 \,.\\
 \end{aligned}
\right.
\end{equation}
\begin{pr}\label{M-R}
Fix $a<\infty$. The linear system \eqref{linear} has a unique solution $u\in \mE_1(a)$ if and only if $f\in \mE_0(a)$, $g\in\mF(a)$, $u_0\in X_1$, and the compatibility conditions \eqref{compatibility linear} are satisfied. Moreover there exist $ \widetilde C=\widetilde C(a)>0$ such that $$
\norm{u}_{\mE_1(a)}\le \widetilde C \big(\abs{u_0}_1+\norm{f}_{\mE_0(a)}+\norm{g}_{\mF(a)}\big).
$$
\end{pr}
We  turn next to  the problem of global in time existence  for the system
\begin{equation}\label{linear stable}
\left\{
 \begin{aligned}
    \partial _t w+P^s AP^s w &= f(t)  & &\mbox{in }\Omega \,,         & &t>0,         \\
     Bw                      &= g(t)  & &\mbox{on }\partial\Omega \,, & &t>0,         \\
     w(0)                    &= w_{0} & &\mbox{in }\Omega \,,
 \end{aligned}
\right.
\end{equation}
where $t \in ( 0, \infty ]$. 
\begin{pr}\label{S-M-R}
Let $0<a\leq\infty$ and  $0<\sigma< \omega $, where $\omega=$ inf $\{$Re$\lambda: \lambda \in \sigma _{s}\}$. The linear problem \eqref{linear stable} has a unique solution $w$ such that $e^{\sigma t}w \in \mathbb{E}_1(a)$ if and only if $e^{\sigma t}f\in C^\frac{\a}{2m}(I_{a};X)\cap B(I_{a}; X_0^s)$, $e^{\sigma t}g\in \mathbb{F}(a)$, $w_0\in X_1^s$, and the compatibility conditions \eqref{compatibility linear} are satisfied. Moreover there exists a constant $C_{0} $, independent of $a$, such that $$
\norm{ e^{\sigma t }w}_{\mE_1(a)}\le C_{0} \Big(\abs{w_0}_1+\norm{e^{\sigma t }f}_{\mE_0(a)}+\norm{ e^{\sigma t } g}_{\mF(a)} \Big) \,.
$$
\end{pr}
\begin{proof}
To show the "only if" part, use the system of equations \eqref{linear stable}. Let us prove the "if" part.  First observe that if  $u$ solves \eqref{linear} with $u_0=w_0$ then  the function $w=P^su$  solves problem \eqref{linear stable}. Now let us denote by  $u_1$ the solution of \eqref{linear} with $A + 1$ replacing $A$. Since $\sigma ( A_0 + 1 ) \subset \mC_+$, applying    Corollary \ref{corollary asymptotic} below, we get a uniform bound for $e^{\sigma t}u_1$ in $\mathbb{E}_1(\infty)$.

Setting $ u_2=u-u_1$ we find that $z=P^s u_2$ solves the problem
\begin{equation}\label{P^s u_2}
\partial_t z+ P^s A P^s z=P^s u_1, \quad Bz=0, \quad z(0)=0.
\end{equation}
Let $u_3$  denote the solution of 
\begin{equation}\label{extra equation}
\partial_t z+  A  z=P^s u_1, \quad Bz=0, \quad z(0)=0.
\end{equation}
By applying Theorem \ref{asymptotic behavior theorem for systems} below to \eqref{extra equation} with $f=P^s u_1,\, u_0=0$, $g=0$  we  find  a uniform bound for $e^{\sigma t}u_3$ in $\mathbb{E}_1(\infty)$ (it is easy to see that  \eqref{condition on u_0 system} holds). Now using the fact that $P^s u_3$ solves \eqref{P^s u_2}
we also obtain  a uniform bound for $e^{\sigma t}P^{s}u_3=e^{\sigma t}P^s u_2$ in $\mathbb{E}_1(\infty)$.
This finishes the proof.
\end{proof}
\textbf{(e)} Let us turn our attention    to the nonlinearities $T$, $R$ and $S$. Here we  derive estimates which are needed for applying Proposition \ref{S-M-R}. 

Let $0<r\leq R,$ and set  
\begin{align*}
K(r)=\text{sup}\{\norm{F^\prime (\varphi) }_{\mathcal{L}( C^{2m+\a}(\o),C^\a(\o))}:\,\varphi \in B(0,r)\subset C^{2m+\a}(\o)\},\\
H_{j}(r)=\text{sup}\{\norm{G_j^\prime (\varphi) }_{\mathcal{L}( C^{2m+\a}(\o),C^{2m+\a-m_j}(\partial\Omega))}:\,\varphi \in B(0,r)\subset C^{2m+\a}(\o)\},
\end{align*}
for $j= 1,\dots,mN$. Since $F^\prime(0)=0$ and $G_j^\prime (0)=0$, $K(r)$ and $H_j(r)$ tend to $0$ as $r\rightarrow 0$. Let $L>0$ be such that, for all $\varphi$, $\psi\in B(0,r)\subset C^{2m}(\o)$ with small $r$,
\begin{eqnarray*}
\norm{F^\prime (\varphi)-F^\prime(\psi) }_{\mathcal{L}( C^{2m}(\o),C(\o))}
\leq L\norm{\varphi-\psi}_{C^{2m}(\o)}, \\
\norm{G_j^\prime (\varphi)-G_j^\prime(\psi) }_{\mathcal{L}( C^{m_j}(\o),C(\partial \Omega))}
\leq L\norm{\varphi-\psi}_{C^{m_j}(\o)}.\\
\norm{G_j''(\varphi)-G_j''(\psi) }_{\mathcal{L}( C^{m_j}(\o),\mathcal{L}(C^{m_j}(\o),C(\partial \Omega))}
\leq L\norm{\varphi-\psi}_{C^{m_j}(\o)}.
\end{eqnarray*}

 In the following, we will always assume that $r \leq \text{min}\{R,\rho_0\}.$
\begin{lem}\label{T}
There exist a constant $C_{1}$ such that 
\begin{align*}
\abs{T(v,w)}\leq C_{1}\abs{w}_{1}
\end{align*}
for any $u\in \overline{B_{X_1}(0,r)}.$ 
\end{lem}
\begin{proof}
From  \eqref{phi} we see
$$
\abs{v+\phi(v)+w}_1=\abs{u}_1\leq r,\quad  \abs{v+\phi(v)}_1\leq \abs{v}_1+\abs{\phi(v)}_1\leq2r, 
$$
and now taking $z_1=v+\phi(v)+w $ and $z_2=v+\phi(v) $ in the definition of $T(v,w)$ we get
\begin{align*}
\abs{T(v,w)}&=\abs{P^c\l F(z_1)-F(z_2)\r}+\abs{P^cAw}\\
&\leq  \abs{F(z_1)-F(z_2)}_{0}+\norm{P^cA}_{\mathcal{L}(X_1, X^c)}\abs{w}_1\\
&\leq (K(2r)+C_{2})\abs{w}_1,
\end{align*}
where $C_{2}:=\norm{P^cA}_{\mathcal{L}(X_1, X^c)}$ which is not necessarily  small.
\end{proof}
\begin{pr}\label{nonlinear estimate}
If $z_1, z_2\in \overline{B_{\mathbb{E}_1(a)}(0,r)}$, $\sigma \geq 0$ then 
\begin{align*}
\begin{array}{c}
\norm{e^{\sigma t}(F(z_1)-F(z_2))}_{\mathbb{E}_0(a)}\leq D(r)\norm{e^{\sigma t}(z_1-z_2)}_{\mathbb{E}_1(a)},\\ \\
\norm{e^{\sigma t} (G(z_1)-G(z_2))}_{\mathbb{F}(a)}\leq D(r)\norm{e^{\sigma t }(z_1-z_2)}_{\mathbb{E}_1(a)},
\end{array}
\end{align*}
where $D(r) \rightarrow 0 $ as $r \rightarrow 0$.
\end{pr}

The proof is given in  the appendix.

\begin{lem}\label{estimate}
If $u \in \overline{B_{\mE_1(a)}(0,r)},$ then $v+\phi(v)\in \overline{B_{\mE_1(a)}(0,4r+L'r^{2})}$.
\end{lem}
\begin{proof}
For $0\leq t\leq  a $, again by  \eqref{phi}, we have
$$\abs{v(t)+\phi(v(t))}_{1}\leq2\abs{v(t)}_1\leq 2\abs{u(t)}_1\leq 2r     
$$
while for $0\leq s\leq t \leq a$,
\begin{align*}
&\abs{v'(t)+\phi'(v(t))v'(t)-v'(s)-\phi'(v(s))v'(s)}\\
&\leq \abs{v'(t)-v'(s)}+\abs{\phi'(v(t))\l v'(t)-v'(s)\r}+\abs{\phi'(v(t))-\phi'(v(s))}_{\mathcal{L}(X^c, X_1)}\abs{v'(s)}\\
&\leq 2(t-s)^\frac{\a}{2m}\norm{v}_{\mE_1(a) }+L' \abs{v(t)-v(s)}\abs{v'(s)} \\
&\leq2(t-s)^\frac{\a}{2m}\norm{v}_{\mE_1(a) }+L'(t-s)^\frac{\a}{2m}\norm{v}^{2}_{\mE_1(a) }\\
&\leq (t-s)^{\frac{\a}{2m}}\l 2r+L'r^{2}\r
\end{align*}
 Note that we have used   \eqref{phi} and Lemma \ref{parabolic estimate} to obtain the second inequality. This completes the proof.  
\end{proof}
\begin{pr}\label{TRS}
If $u\in \overline{ B_{\mathbb{E}_1(a)}(0,r)}$, $\sigma \geq 0$ then 
\begin{align*}
(i)\quad&\norm{e^{\sigma t} T(v,w)}_{\mathbb{E}_0(a)}\leq C_3\norm{e^{\sigma t} w}_{\mathbb{E}_1(a)},&\\
(ii)\quad&\norm{e^{\sigma t} R(v,w)}_{\mathbb{E}_0(a)}\leq C(r)\norm{e^{\sigma t} w}_{\mathbb{E}_1(a)},&\\
(iii)\quad&\norm{e^{\sigma t}S(v,w)}_{\mathbb{F}(a)}\leq C(r)\norm{e^{\sigma t} w}_{\mathbb{E}_1(a)},&
\end{align*}
where $C(r) \rightarrow 0$ as $r$ goes to zero.
\end{pr}
\begin{proof}
Let us prove (i). Setting $z_1:=u=v+\phi(v)+w $ and $z_2:=v+\phi(v) $  by Lemma \ref{estimate} we have $$
\norm{z_1}_{\mE_1(a)},\norm{z_2}_{\mE_1(a)}\leq 4r+L'r^{2}\,.
$$
Hence we can now apply Proposition  \ref{nonlinear estimate} to conclude
\begin{align*}
\norm{e^{\sigma t}T(v,w)}_{\mE_0(a)}&\leq \norm{P^c (e^{\sigma t}(F(z_1)-F(z_2)))}_{\mE_0(a)}+\norm{e^{\sigma t}P^cAw}_{\mE_0(a)}\\
&\leq  \norm{e^{\sigma t}(F(z_1)-F(z_2))}_{\mE_0(a)}+\norm{e^{\sigma t}P^cAw}_{\mE_0(a)}\\  
&\leq D(4r+L'r^{2}) \norm{e^{\sigma t}w}_{\mathbb{E}_1(a)}+ \norm{e^{\sigma t} P^c A w}_{\mE_0(a)} \,.
\end{align*}
Now let us consider $\norm{e^{\sigma t}P^cAw}_{\mE_0(a)}$.

For $0\leq t\leq  a $,$$
\abs{e^{\sigma t}P^cAw(t)}\leq C_2\abs{e^{\sigma t}w(t)}_1\leq C_2\norm{e^{\sigma t}w}_{\mE_1(0,a)}
$$
while for $0\leq s\leq t \leq a$,
\begin{align*}
\abs{ e^{\sigma t}P^cAw(t)-e^{\sigma t} P^cAw(s)}&\leq \abs{A(e^{\sigma t }w(t)-e^{\sigma s}w(s))}\leq \norm{e^{\sigma t }w(t)- e^{\sigma s}w(s)}_{D(A)}\\
&\leq\norm{e^{\sigma t}w(t)- e^{\sigma s}w(s)}_{C^{2m}(\o)}\\
&\leq C'(t-s)^{\frac{\a}{2m}} \norm{e^{\sigma t}w}_{\mE_1(a)}\,,
 \end{align*}
 where we have used Lemma \ref{parabolic estimate} to obtain the last inequality and $C'$ is the corresponding embedding constant. Setting $C_3:=D(4r+L'r^{2})+C'+C_2$\ we complete the proof of (i).

We now prove (ii). Similarly as in (i) we get for the first term in  $e^{\sigma t}R(v,w)$ 

\begin{equation*}
\norm{P^s (e^{\sigma t}(F(z_1)-F(z_2)))}_{\mE_0(a)}\leq D(4r+L'r^{2}) \norm{e^{\sigma t}w}_{\mathbb{E}_1(a)}\,. 
\end{equation*}
Let us estimate the second term in $R(v,w)$ namely, $e^{\sigma t} \phi'(v) T(v,w)$.

For $0\leq t\leq a,$ by \eqref{phi} and Lemma \ref{T} we have 
\begin{align*}
\abs{e^{\sigma t}\phi'(v(t))T(v(t),w(t))}_0&\leq \abs{e^{\sigma t}\phi'(v(t))T(v(t),w(t))}_1\\
&\leq \norm{\phi'(v(t))}_{\mathcal{L}(X^c, X_1^s)}\abs{e^{\sigma t}T(v(t),w(t))}\\
&\leq \eta(r)\abs{e^{\sigma t}T(v(t),w(t))}\leq C_1\eta(r)\abs{e^{\sigma t}w(t)}_1 \\
&\leq C_1 \eta(r)\norm{e^{\sigma t}w}_{\mE_1(a)}\,,
\end{align*}
while for $0\leq s\leq t \leq a$,
\begin{align*}
&\abs{e^{\sigma t}\phi'(v(t))T(v(t),w(t))-e^{\sigma s}\phi'(v(s))T(v(s),w(s))}\\ 
&\leq\norm{\phi'(v(t))}_{\mathcal{L}(X^c, X_1)}\abs{e^{\sigma t}T(v(t),w(t))-e^{\sigma s}T(v(s),w(s))}\\
& \quad +\norm{\phi'(v(t))-\phi'(v(s))}_{\mathcal{L}(X^c, X_1)}\abs{e^{\sigma s}T(v(s),w(s))}\\
&\leq(t-s)^{\frac{\a}{2m}}\eta(r)\norm{e^{\sigma t}T(v,w)}_{\mE_0(a)}+C_1
\norm{\phi'(v(t))-\phi'(v(s))}_{\mathcal{L}(X^c, X_1)}\abs{e^{\sigma t }w(t)}_1\\
&\leq (t-s)^{\frac{\a}{2m}}\eta(r)C_3\norm{e^{\sigma t }w}_{\mathbb{E}_1(a)}+C_1
L'\abs{v(t)-v(s)}\norm{e^{\sigma t}w}_{\mE_1(a)}\\
&\leq (t-s)^{\frac{\a}{2m}}\eta(r)C_3\norm{e^{\sigma t}w}_{\mathbb{E}_1(a)}+
(t-s)^{\frac{\a}{2m}}C_1L'\norm{v}_{E_1(a)}\norm{e^{\sigma t}w}_{\mE_1(a)}\\
&\leq (t-s)^{\frac{\a}{2m}}\l\eta(r)C_3+
C_1L'r\r\norm{e^{\sigma t}w}_{\mE_1(a)} \,.
\end{align*}
Finally by defining $C(r):=\eta(r)C_3+
C_1L'r+C_1 \eta(r)+D(4r+L'r^{2})$  we complete the proof of (ii). Using Proposition \ref{nonlinear estimate} and Lemma \ref{estimate}, we easily get (iii).
\end{proof}
\textbf{(f)  }We consider now an existence theorem for problem \eqref{eq1}. As a first step, we  show existence for large time using the contraction mapping principle. 
\begin{pr}\label{existence at large}
  For every $T>0$, there are $r > \rho > 0$ such that \eqref{eq1} has a solution $u\in \mathbb{E}_1(T)$ provided $\abs{u_0-u_*}_1\leq\rho.$ Moreover, $u$ is the unique solution in $B_{\mathbb{E}_1(T)}(0,r)$.
\end{pr}
\begin{proof}
The proof is almost exactly the same as the one in Theorem 4.1 in \cite{Lunardi2002385}. However
for the convenience of the reader we provide the details.

Let $ 0 < r \leq R $ and define a nonlinear map 
\begin{align*}
\Gamma : \left\{ 
                 w \in \overline{B(0,r)} \subset \mE_1(T): w(\cdot,0) = u_0
\right\}\longrightarrow \mE_1(T) \, ,
\end{align*}
by $\Gamma w =v $, where $v$ is the solution of 
\begin{equation*}
\left\{
 \begin{aligned}
    \partial_t v + A v &= F(w) \,, & &\text{ in } \o \times [0,T] \,, \\ 
    Bv                 &= G(w) \,, & &\text{ on } \partial \Omega \times         [0,T] \,, \\
    v|_{t=0}           &=u_0 \,,   & &\text{ in } \o \,.
 \end{aligned}
\right.
\end{equation*}
Proposition \ref{M-R}, gives the estimate 
$$
\norm{v}_{\mE_1(T)}\le \widetilde C \big(\abs{u_0}_1+\norm{F(w)}_{\mE_0(T)}+\norm{G(w)}_{\mF(T)}\big) \,,
$$
with $\widetilde C=\widetilde C(T)$ in which we could assume without loss of generality that $\widetilde C>1$. Hence by Proposition \ref{nonlinear estimate} we have 
$$
\norm{\Gamma (w)}_{\mE_1(T)} \leq \widetilde C \left(\abs{u_0}_1+2D(r)\norm{w}_{\mE_1(T)} \right).
$$ 
Consequently, if $r$ is so small that 
\begin{equation}\label{estimate for r}
2\widetilde C D(r) \leq \frac{1}{2}\,,
\end{equation}
and $u_0$ is so small that 
\begin{equation}\label{estimate for initial data}
\abs{u_0}_1 \leq \frac{r}{2\widetilde C}\,,
\end{equation}
 then $\Gamma$ maps the ball $\overline{B(0,r)}$ into itself. We then continue to show that $\Gamma$ is a $\frac{1}{2}$-contraction. To prove this let $w_1$, $w_2 \in B(0,r)$. Then 
$$
\norm{\Gamma w_1 -\Gamma w_2} _{\mE_1(T)} \leq \widetilde C\Big( \norm{F(w_1)-F(w_2)}_{\mE_0(T)}+
\norm{G(w_1)-G(w_2)}_{\mF(T)}\Big)\,, 
$$ 
and  using again  Proposition \ref{nonlinear estimate} we have
\begin{align*}
\norm{\Gamma w_1 -\Gamma w_2} _{\mE_1(T)} & \leq 2\widetilde C D(r)\norm{w_1-w_2}_{\mE_1(T)}\\
& \leq \frac{1}{2} \norm{w_1-w_2}_{\mE_1(T)}\,.
\end{align*}
Now the statement follows by the contraction mapping principle.
\end{proof}

Let us next observe from the inequalities \eqref{estimate for r} and \eqref{estimate for initial data} that  for a given time $T$, we can take $r$ as small as  we want provided $\rho < r $ is small enough. Now the strategy for proving the global existence  is as follows:  

We fix a time $T$ and choose $r \leq \text{min}\{R,\rho_0\}  $  small enough such that
\begin{equation}\label{constant}
2C_{0}C(r)\leq\frac{1}{2}
\end{equation}
and that  \eqref{estimate for r} holds. For such an $r$  we have a corresponding $\rho < r$ (by Proposition \ref{existence at large}). In conclusion  by Proposition \ref{existence at large}, problem \eqref{eq1} admits for $u_0\in B_{X_1}(0,\rho)$ a unique solution $
u\in \overline{B_{\mathbb{E}_1(T)}(0,r)}$.  The strategy is now as follows: We will find  some $\delta < \rho$ such that the solution $u(t)$ of  \eqref{eq1} with initial value $u_0 \in B_{X_1} ( 0 , \delta )$ defined on its maximal interval of existence always stays in the ball $B_{X_1} ( 0 , \rho)$. Then by Proposition \ref{existence at large}  we see immediately  that the maximal interval  of existence can not be  bounded  and this proves the global existence. 

Arguing as above there exists some $\delta'< \frac{\rho}{2}$ such that  the problem \eqref{eq1} admits for $u_0\in B_{X_1}(0, \delta')$ a unique solution 
\begin{equation}\label{delta'}
   u \in \overline{B_{\mathbb{E}_1(T)}(0,\frac{\rho}{2})}\,.
\end{equation}
Suppose that $u_0 \in B_{X_1}(0,\delta )$, where $\delta \leq \delta' < \rho$ is a number to be selected  later. Let  $[0 , t_*)$   be the  maximal interval of existence   of the solution $u(t)$ of \eqref{eq1} with initial value $u_0$. Furthermore let   $t_1$  be the existence time for the ball $B_{X_1}(0,\rho)$, i.e., 
$$
t_1:=\text{sup} \{t\in (0,t_*)\, : \abs{u(\tau)}_1\leq \rho, \quad \tau\in[0,t]\}.
 $$
Suppose also  $t_1<t_*$. Note that $t_1 \geq T$ by \eqref{delta'}.  
\begin{lem}\label{myself}
Under the conditions above we have $\norm{u}_{\mathbb{E}_1(t_1)}\leq r.$
\end{lem}
\begin{proof}
Since $\abs{u_{0}}_1<\delta' < \rho$,   we have $u\in B_{\mathbb{E}_1(T)}(0,r) $ by Proposition \ref{existence at large}. By the definition of $t_1$ and the fact that $T \leq t_1$ we get $\abs{u(T)}_1\leq \rho$. Therefore we can now start with the initial data $u(T)$ and  by finitely often repeating the same process we complete the proof (since  $T$ is constant, we will get  $u \in B_{\mathbb{E}_1(kT)}(0,r)$ for some $k$ such that $k T > t_1$ and therefore the estimate follows immediately).\end{proof}
 Now we apply \eqref{normal form}, Proposition \ref{S-M-R}, Lemma \ref{myself} and Proposition \ref{TRS} and derive
\begin{align*}
\norm{e^{\sigma t}w}_{\mE_1(t_{1})}&\leq  C_{0} \big(\abs{w_0}_1+\norm{e^{\sigma t}R(v,w)}_{\mE_0(a)}+\norm{e^{\sigma t}S(v,w)}_{\mF(a)}\big)\\
&\leq C_0 \abs{w_0}_1+2C_{0}C(r)\norm{e^{\sigma t}w}_{\mE_1(t_1)} \,. 
\end{align*}
Together with \eqref{constant} this implies\begin{equation}\label{w}
\norm{e^{\sigma t}w}_{\mE_1(t_{1})}\leq 2C_0\abs{w_0}_1,\quad \sigma\in[0,\omega).
\end{equation}
Hence for $t\in [0,t_1]$ 
$$
\abs{e^{\sigma t}w(t)}_1\leq \norm{e^{\sigma t}w}_{\mE_1(t_1)}\leq 2C_0 \abs{w_0}_1 \,,
$$
and so \begin{align}\label{exponential}
\abs{w(t)}_1\leq 2C_0e^{-\sigma t} \abs{w_0}_1,\quad t\in [0,t_1],\quad \sigma\in[0,\omega).
\end{align}
Using the equation for $v$ in \eqref{normal form} and Lemma \ref{T} we obtain
\begin{align*}
\abs{v(t)}&\leq \abs{v_0}+\int_0^t \!\abs{T(v(s),w(s)}\, \mathrm{d}s\\
&\leq \abs{v_0}+C_1\int_0^t\abs{w(s)}_1\, \mathrm d s\\
&\leq \abs{v_0}+C_1\int_0^\infty \ e^{-\sigma s }\, \mathrm d s \: \norm{e^{\sigma t }w}_{\mE_1(t_1)}\\
&\leq \abs{v_0}+\frac{C_1}{\sigma} \norm{e^{\sigma t }w}_{\mE_1(t_1)}\\
&\leq \abs{v_0}+ C_{4}\abs{w_0}_1,\quad t\in [0,t_1],
\end{align*}
where $C_4=2C_0\frac{C_1}{\sigma}$. Combining the last two estimates and taking into account  \eqref{phi} we find  
$$
\abs{u(t)}_1\leq C_5 \abs{u_0}_1,\quad t\in[0,t_1].
$$
for some constant $C_5\geq 1$.
In particular  this inequality  is satisfied for $t=t_1.$  Therefore   choosing $\delta\leq\frac{\delta'}{2C_5}$, we find $\abs{u(t_1)}_1\leq \delta'/2$. But this is a contradiction to the definition of $t_1$ since by \eqref{delta'} we could start with  $t_1 $ and continue further and still being in the ball $B_{X_1}(0,\rho)$ hence $t_1=t_*$. By Lemma \ref{myself}   we get then uniform bounds $\norm{u}_{\mathbb{E}_1(a)}\leq r,$ for all $a<t_*$. As a result of Proposition \ref{existence at large}, we conclude $t_*=\infty.$ 
\\[0.0 cm]

\textbf{(g)} Finally, we repeat the estimates above on the interval $[0,\infty)$. This yields
\begin{align*}
\abs{v(t)}\leq \abs{v_0}+ C_{4}\abs{w_0}_1,\quad  \abs{w(t)}_1\leq 2C_0e^{-\sigma t} \abs{w_0}_1,\quad t\in [0,\infty),
\end{align*}
for $u_0\in B_{X_1}(0,\delta).$ Furthermore, the limit 
$$\text{lim}_{t\rightarrow\infty}v(t)=
\text{lim}_{t\rightarrow\infty} \left( v_0+\int_0^\infty \!T(v(s),w(s))\,\mathrm{d}s\right)\,=:v_\infty $$
exists in $X$ since the integral converges absolutely. Hence  
$$
u_\infty:= \text{lim}_{t\rightarrow\infty}u(t)=\text{lim}_{t\rightarrow\infty} v(t)+\phi (v(t))+w(t)=v_\infty+\phi (v_\infty).
$$
exists too and $u_\infty $ is a stationary solution of \eqref{eq1} due to \eqref{equilibria}. Moreover, Lemma \ref{T} and \eqref{w} imply
 
\begin{align*}
\abs{v(t)-v_\infty}&=\left|\int_t^\infty T(v(t),w(t))\,\mathrm d s\right|\\
&\leq C_1 \int_t^\infty \abs{w(s)}_1\, \mathrm d s\\
&\leq  C_1 \int_t^\infty e^{-\sigma s }\, \mathrm d s\: \norm{e^{\sigma t  }w}_{\mE_1(\infty)}\\
&\leq C_4 e^{-\sigma t}\abs{w_0}_1 \quad t\geq 0.
\end{align*}
Hence $    v(t) \to  v_\infty$ in $X$ at an exponential rate as $t \to \infty$. Finally, using  \eqref{phi} and \eqref{exponential} we conclude
\begin{align*}
\abs{u(t)-u_\infty}_1&=\abs{v(t)+\phi(v(t))+w(t)-u_\infty}_1\\
&\leq \abs{v(t)-v_\infty}+\abs{\phi(v(t))-\phi(v_\infty)}_1+\abs{w(t)}_1\\
&\leq (2C_4+2C_0)e^{-\sigma t } \abs{w_0}_1\\
&\leq Ce^{-\sigma t } \abs{P^s u_0-\phi(P^c u_0)}_1,
\end{align*}
which proves the   second part of Theorem \ref{theo01}. Note that by Lemma \ref{myself} it follows that by choosing $ 0<\delta \leq \rho$ sufficiently small, the solution  starting in $B_{X_1}(u_*,\delta)$  exists for all times and stays within $B_{X_1}(u_*,r)$. This implies stability of $u_*$.
\end{proof}
\begin{rem}
Note that the assumption $u_\ast \equiv 0$ in Theorem \ref{theo01} is not a restriction. The case of a general stationary solution $\bar{u}$ can be reduced to the case of the zero stationary solution by considering
$$
U(t)=u(t)-\bar{u}.
$$
\end{rem}

%% file: lens-shaped.tex
\section{Stability of lens-shaped networks under  \\surface
diffusion flow}\label{application}
In this section we show that the   lens-shaped networks generated by circular arcs are stable under the surface diffusion flow. We will see in this section how well the    generalized principle of   linearized stability in the parabolic Hölder spaces, i.e., Theorem \ref{Main Result}  can be used as a tool to show the stability.  Indeed, the set of equilibria forms a finite-dimensional smooth manifold and the resulting PDE has nonlocal terms in the highest order derivatives. 
\subsection{The geometric setting}
The surface diffusion flow is a geometric evolution equation for an evolving hypersurface  $\Gamma=\{\Gamma(t)\}_{t>0}$ in which
\begin{equation}\label{surface diffusion}
V=-\Delta_{\Gamma(t)} \kappa\,,
\end{equation}
where $V$ is the normal velocity, $\kappa$ is  the sum of the principle curvatures, and $\Delta_{\Gamma(t)}$ is the Laplace-Beltrami operator of the hypersurface $\Gamma(t)$. Our sign convention is that $\kappa$ is negative for spheres for which we choose the outer unit normal. 

Constant-mean-curvature surfaces  are stationary solutions of \eqref{surface diffusion}. Now, it is natural to ask whether  these solutions  are stable under the flow.  Indeed, Elliott and Garcke \cite{Elliott-Garcke} showed the stability of circles in the plane and   one year later Escher, Mayer and Simonett \cite{Escher-Mayer-Simonett} proved the stability of spheres in higher dimensions. In general, the surfaces  will meet an outer boundary or they might intersect at triple or multiple lines.

 A lens-shaped network consists of two smooth curves and two rays arranged as in Figure \ref{fig:lens-shaped}, that is to say, we assume that the network  has reflection symmetry across the $x^1$-axis and that the two curves meet the two rays  with a constant angle $\pi-\theta$, where $0<\theta<\pi$. Note that $\theta=\frac{\pi}{3} $ corresponds to  symmetric angles at the triple junction. 

\begin{figure}[htbp]
        \centering
        \begin{tikzpicture}[scale=0.7,>=stealth]
\draw (-4,0)-- (-7,0);   
\draw (4.5,0)-- (7.5,0);  
\draw (5,0) arc (50:94:1)node[above=7pt,right=2pt] {$\pi-\theta$};       \draw (-4.46,0) arc (-50:-93:-1)node[above=7pt,right=-25pt] {$\pi-\theta$};
   
\draw (-4,0) .. controls (-3,1) and (-2,2) .. (0,1.2);
\draw (0,1.2) .. controls (1,0.7) and (2,1.2) .. (2,1.2);
\draw (2,1.2) .. controls (3,2) and (4,0.5) .. (4.5,0);
\draw (-4,0) .. controls (-3,-1) and (-2,-2) .. (0,-1.2);
\draw (0,-1.2) .. controls (1,-0.7) and (2,-1.2) .. (2,-1.2);
\draw (2,-1.2) .. controls (3,-2) and (4,-0.5) .. (4.5,0);
        \end{tikzpicture}
        \caption{ Lens-shaped network}
        \label{fig:lens-shaped}
             \end{figure}
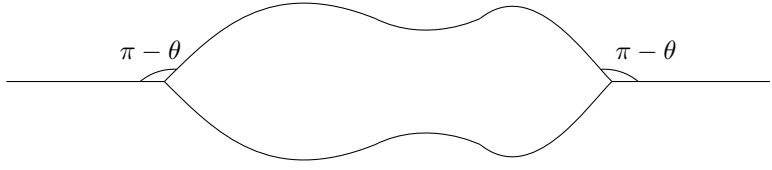

More precisely, a  lens-shaped network is determined by a curve $\Gamma$ with the following property:
\begin{align*}
\begin{cases}
\partial \Gamma \subset \{(x,y)\in \mathbb R^2: y=0\}\,, \\
\sphericalangle(N,e_2)|_{\partial \Gamma}=\theta \,,
\end{cases}
\end{align*}
where $N$ is the unit normal to $\Gamma$ pointing outwards of the bubble, see e.g. Figure \ref{fig:spherical arc}. 

Then the entire   lens-shaped network is defined by four curves:
$\Gamma_1,\Gamma_2,\Gamma_3,\Gamma_4$, where $\Gamma_1$ is the curve $\Gamma$
describe  above, $\Gamma_2$ is the reflection of $\Gamma_1$ across the $x^1$ -axis and   $\Gamma_3$, $\Gamma_4$ are the rays contained in the $x^1$-axis
meeting $\Gamma_1$ and $\Gamma_2$ at triple junctions.

We study the following problem introduced by Garcke and Novick-Cohen \cite{GarckeNovick}: Find evolving  lens-shaped networks $\Gamma_1(t),\dots, \Gamma_4(t)$  as described above with the following properties:
\begin{equation}\label{surface diffusion for lens}
\left \{
 \begin{aligned}
    V_i                                               &= - \Delta _{\Gamma_i}       \kappa_i                               & &\text{ on } \Gamma_i(t),           & &t>0, & &(i=1,2,3,4),\\
    \nabla_{\Gamma_1} \kappa_1 \cdot n_{\partial \Gamma_1} &=\nabla_{\Gamma_2}       \kappa_2 \cdot n_{\partial \Gamma_2} \\
                                                      &=\nabla_{\Gamma_i}       \kappa_i \cdot n_{\partial \Gamma_i}  & &\text{ on } \partial \Gamma_i(t),       & &t>0, & &(i=3,4),\\
    \Gamma_i(t)|_{t=0}                                &=\Gamma^0_i & &&&&&(i=1,2,3,4),
 \end{aligned}
\right.
\end{equation}
where $\Gamma_i^0$ $(i=1,2,3,4)$ form  a  given initial  lens-shaped network fulfilling the balance of flux condition, i.e., the second condition in \eqref{surface diffusion for lens}. Here $V_i$ and $\kappa_i$ are the normal velocity and mean curvature of $\Gamma_i(t)$, respectively,    $n_{\partial \Gamma_i}$ is the outer unit conormal of $\Gamma_i$ at boundary points and $\nabla_{\Gamma_i}$ denotes  the surface gradient of the curve $\Gamma_i(t)$.
 
We choose the unit normal  $N_2(\cdot,t)$ of $\Gamma_2(t)$  to be  pointed inwards of the bubble. Then with this choice of normals we observe that $\kappa_2=-\kappa_1$ at the boundary points and therefore    we get 
$$
\kappa_1+\kappa_2+\kappa_i=0\quad \text{ on } \partial \Gamma_i(t) \quad \text{ for } i=3,4\,,
$$
which must hold  at the triple junctions for more general triple junctions (non-symmetric, non-flat) with $120$ degree angles. We refer to Garcke, Novick-Cohen
\cite{GarckeNovick} for the precise setting of the general problem. 

 Let us note that solutions to \eqref{surface diffusion for lens} preserve the enclosed area. Indeed, by Lemma 4.22 in \cite{Depner}, we have 
\begin{align*}
\frac{d}{dt}\int_{\Omega(t)}1\,\mathrm dx&= -\int_{\Gamma_2(t)}V_2 \,\mathrm d s+\int_{\Gamma_1(t)}V_1 \, \mathrm d s \\
&=\int_{\Gamma_2(t)}\Delta_{\Gamma_2} \kappa_2 \mathrm d s-\int_{\Gamma_1(t)}\Delta_{\Gamma_1} \kappa_1 \mathrm d s \\
&= \int_{\partial \Gamma_2(t)} \nabla_{\Gamma_2} \kappa_2 \cdot n_{\partial \Gamma_2}\, \mathrm d s  -\int_{\partial \Gamma_1(t)} \nabla_{\Gamma_1} \kappa_1 \cdot n_{\partial \Gamma_1}\, \mathrm d s\\
&=0\,,
\end{align*}
where $\Omega(t)$  is defined as the region bounded by $\Gamma_1(t)$ and $\Gamma_2(t)$.

Using the fact that  the curvature of $\Gamma_3(t)$ and $\Gamma_4(t)$ are zero, it is easy to verify that the family of  lens-shaped networks  $\Gamma_1(t),\dots, \Gamma_4(t)$ evolves according to \eqref{surface diffusion for lens} if $\Gamma(t):=\Gamma_1(t)$ satisfies 

\begin{equation}\label{cap problem}
\left \{
 \begin{aligned}
    & V =-\Delta_\Gamma \kappa                                      & &\text{         on } \Gamma(t)\,,           & &t>0 \,,\\
    & \partial \Gamma(t) \subset \{(x,y)\in \mathbb R^2: y=0\} & &                                   & &t>0 \,,\\
    & N\cdot e_2=\cos \theta\                                  & &\text {         on }{\partial \Gamma(t)}\,, & &t>0 \,,\\
    &\nabla_{\Gamma} \kappa \cdot n_{\partial \Gamma}=0             & &\text{         on } \partial\Gamma(t)\,,   & &t>0 \,,\\
    &\Gamma(t)|_{t=0}=\Gamma_0\,,
 \end{aligned}
\right.
\end{equation}
where   $\Gamma_0$ is a given initial curve which fulfills the contact, angle and  no-flux condition as above.
\begin{rem}
The equation $V=-\Delta_{\Gamma}\kappa$ written in a local parameterization is a fourth-order parabolic equation and above we prescribe three boundary conditions. This is due to the fact that \eqref{cap problem} is a free boundary problem because the points in $\partial \Gamma$ can move in the set $\{y=0\}$, see \cite{BaconneauLunardi} for a related second-order problem. Moreover, we would like to refer to the work of Schnürer and co-authors \cite{SchnurerCoauthors}, where they consider  the evolution of symmetric convex lens-shaped networks under the curve shortening flow. 
\end{rem}  
Let us look at  equilibria of the problem \eqref{cap problem}.  It is easy to verify that the curvature of the stationary solutions  is constant and so the set of  the stationary solutions of \eqref{cap problem} consists precisely of all  circular arcs  that intersect the $x$-axis with $\pi-\theta$ degree angles  denoted by $CA_r(a_1,-r\cos \theta)$, where $\abs r$  is the radius and $(a_1, -r\cos \theta)$ are the coordinates of the center with $a_1 \in \mathbb R$, $ r\in \mathbb{R}\setminus \{ 0 \}$ (see Figure \ref{fig:spherical arc} for the justification of the coordinates of the center).
Therefore the set of equilibria forms a 2-parameter family, the parameters are the  radius of the circular arc  and the first component of  the center.

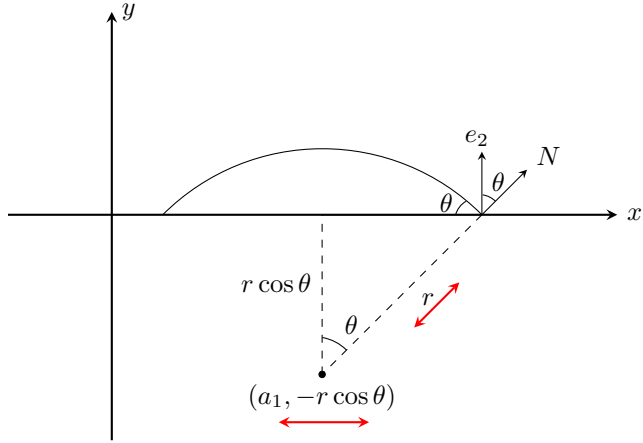
\begin{figure}[htbp]
        \centering
        \begin{tikzpicture}[scale=0.6,>=stealth]
                \draw (0,0) arc (45:135:5);
               \draw [->,thick](-10.5,0)--(3,0)node[right]{$x$};
               \draw [->,thick](-8.2,-5)--(-8.2,4.5)node[right]{$y$};
               \filldraw [black] (-3.54,-3.54) circle (2pt)node[below]{$(a_1,-r\cos \theta)$};
               \draw[dashed] (0,0) -- (-3.54,-3.54)node[below=-1cm,right=1.2cm]{$r$};
               \draw [->](0,0) -- (1,1)node[below=-5pt,right=0cm]{$N$};
               \draw [->](0,0) -- (0,1.4)node[below=-5pt,right=-10pt]{$e_2$};
                \draw (0.3,0.3) arc(45:82:0.5)node[below=-5pt,right=0pt]{$\theta$};
               \draw (-0.35,0.3) arc(120:165:0.5)node[below=-4.7pt,right=-9.5pt]{$\theta$};
               \draw [dashed] (-3.54,-3.54) -- (-3.54,0)node[below=0.9cm, right=-1.2cm]{$r\cos \theta$};
              \draw (-3,-3) arc(45:80:1)node[below=-5pt,right=5pt]{$\theta$}; 
               \draw (-3.54,0)--(0,0);
               \draw [<->,red,thick](-4.5,-4.6) --(-2.5,-4.6);
               \draw [<->,red,thick](-1.5,-2.5)--(-0.5,-1.5);
                      \end{tikzpicture}
                       \caption{Circular arcs $CA_r(a_1,-r\cos \theta)$ for $r>0$}
               \label{fig:spherical arc}
\end{figure}

It is a goal of this section to prove the stability of such stationary solutions (see Theorem \ref{main theorem example}) using the generalized principle of linearized stability in parabolic Hölder spaces, i.e., Theorem \ref{theo01}.

Let us briefly outline how we proceed. At first we parameterize the curves around a stationary curve with the help of a modified distance function introduced in Depner and Garcke \cite{DepnerGarckelinearized}. Note that the linearization in the case of a triple junction with boundary contact is calculated  in \cite{DepnerGarckelinearized} and  the calculations can be easily modified
to the present situation.   We then formulate the evolution problem with the help of this parameterization and derive a highly nonlinear, nonlocal problem \eqref{nonlinear, nonlocal}.  

In Section \ref{linearization and abstract setting}, after deriving  the linearization around the stationary solution, we see   how our nonlinear, nonlocal problem    fits well into our general  evolution  system \eqref{eq1}.  We then continue by checking  the assumption (H1), (H2), (LS),  (SP) and the normality condition \eqref{normality condition}.

Finally, in order to apply Theorem \ref{theo01}, it remains to check the assumption that the stationary solution is normally stable which is done in Section \ref{check normally stability}.

\subsection{Parameterization and PDE formulation}\label{PDE and parametrization}
\subsubsection{Parameterization}
In this section we introduce the mathematical setting in order to reformulate our geometric evolution law, i.e., \eqref{cap problem} as a partial differential equation for an unknown function defined on a fixed domain. To this end, we use a parameterization with two parameters corresponding to a movement in tangential and normal direction, introduced in Depner and Garcke \cite{DepnerGarckelinearized}, see also \cite{DepnerGarckeKohsaka}.

Let us describe $\Gamma(t)$ with the help of a function $\rho: \Gamma_*  \times [0,T)\rightarrow \mathbb R$  as graphs over some fixed stationary solution  $\Gamma_*$. Note that the curvature $\kappa_*$ of $\Gamma_*$ is  constant and negative and  the length of $\Gamma_*$ is $2l_*$, where
 $$-\kappa_* l_*=\theta\,.$$ 
 
Let $x$ be the arc-length parameter of $\Gamma_*$. Then  an arc-length parameterization of $\Gamma_*$ is defined as 
$$
\Gamma_*=\left\{\Phi_*(x): x \in [-l_*,l_*]\right\}.
$$ 
For $\sigma\in \Gamma_*$, we set $\Phi_*^{-1}(\sigma)=x(\sigma)\in \mathbb R$. From now on, for simplicity, we set 
\begin{equation} \label{x and sigma}
\partial _\sigma w ( \sigma ):=\partial_x(w\circ \Phi_*) (x), \quad \sigma = \Phi_* (x),
\end{equation}
 i.e., we omit the parameterization. In particular, we use the  slight abuse of the notation  
\begin{equation} \label{abuse of notation}
w(\sigma)=w(x)\quad (\sigma \in \Gamma_*)\,.
\end{equation}
In order to parameterize a curve close to $\Gamma_*$, we define 
\begin{align}
\Psi:\Gamma_*\times (-\epsilon,\epsilon)\times (-\delta,\delta)&\longrightarrow \mathbb R^2\,,\\
(\sigma,w,r)&\mapsto \Psi(\sigma,w,r):=\sigma+w N_*(\sigma )+r \tau_*(\sigma)\,,\nonumber
\end{align}
where $\tau_*$ is a tangential vector field on $\Gamma_*$ with support in a neighborhood of $\partial \Gamma_*$, which equals the outer unit  conormal $n_{\partial \Gamma_*}$ at $\partial \Gamma_*$.

We define $\Phi=\Phi_{\rho,\mu}$ (we often omit the subscript $(\rho,\mu)$ for shortness) by  \begin{equation} \label{par2}
\Phi:\Gamma_* \times [0,T) \rightarrow \mathbb R^2 \,, \quad \Phi(\sigma,t):=
\Psi(\sigma, \rho(\sigma,t),\mu(\mathrm{pr}(\sigma),t))\,,
\end{equation}
where 
\begin{equation} \label{par1}
\rho:\Gamma_* \times [0,T) \rightarrow (-\epsilon, \epsilon)\,, \quad \mu:\partial \Gamma_*=\{a_*,b_*\}\times[0,T)\rightarrow (-\delta,\delta)\,.
\end{equation}
The projection $\mathrm{pr}:\Gamma_* \rightarrow \partial \Gamma_*=\{a_*,b_*\}$ is defined by imposing the following condition: The point $\mathrm{pr}(x ) \in \partial \Gamma_*$ has the shortest distance on $\Gamma_*$ to $\sigma $. Of course, in a small neighborhood of $\partial \Gamma_*$, the projection $\mathrm{pr}$ is well-defined and smooth. And this is enough for our purpose since we need this projection  just near $\partial \Gamma_*$ because it is used in the product $\mu(\mathrm{pr}(\sigma ),t)\tau_*(\sigma )$, where the second term vanishes outside a (small) neighborhood of $\partial \Gamma_*$. Finally, by setting for small $\epsilon, \delta>0$ and fixed $t$ 
$$
(\Phi)_t: \Gamma_* \rightarrow \mathbb R^2, \quad (\Phi)_t(\sigma ):=\Phi(\sigma ,t) \quad \forall \sigma \in \Gamma_* \,,
$$
 we define a new curve through 
\begin{equation}\label{Gamma rho}
\Gamma_{\rho,\mu}(t):=\mathrm{image}((\Phi)_t)\,.
\end{equation}
Note that  for $\rho \equiv 0 $ and $\mu \equiv 0$ the resulting curve  coincides with a  stationary curve $\Gamma_*$.

As  Figure \ref{Fi:parameter} nicely illustrates, apart from the normal movement, close to  the boundary points the parameter $\mu$ allows for  tangential movement. Therefore the resulting curve not only have the possibility to meet the $x$-axis at its boundary points  but also have the opportunity to be parameterized as a graph over the fixed  stationary curve $\Gamma_*$. The price to pay is the appearance of nonlocal terms  explained explicitly below.    
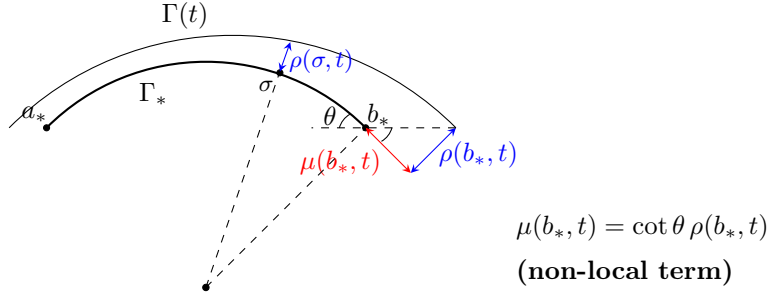
\begin{figure}[htbp]
        \centering
        \begin{tikzpicture}[scale=0.6,>=stealth]
               \draw[thick] (0,0) arc (45:135:5);
               \draw (2,0) arc (45:135:7);                
               \filldraw [black] (-3.54,-3.54) circle (2pt)node [below =-0.8cm , right = 4cm] {$\mu( b_*, t ) = \cot \theta  \, \rho(b_* , t)$} node [below =-0.2cm , right = 4cm]{\textbf{(non-local term)}} ;
               \filldraw [black] (-1.9,1.22) circle (1.9 pt) node[above =0.15 cm , right = 0cm , blue]{\small{$\rho (\sigma , t)$}} ;
               \draw[dashed] (-3.54,-3.54) -- (-1.9,1.25);
               \draw [blue, <->](-1.9,1.25) -- (-1.67,1.9);
               \filldraw [black] (0,0) circle (1.9pt) node[above=5, right = -3]{$b_*$};
                \filldraw [black] (-7.07,0) circle (1.9pt) node[above=5, right = -13]{$a_*$}; 
               \draw [red , <->](0,0) -- (1,-1)node[above=4pt,right=-1.6cm]{$\mu ( b_*, t )$};
               \draw [blue, <->](1,-1) -- (2,0)node[below=10pt,right=-0.35cm]{$\rho ( b_*, t)$};
               \draw[dashed] (-3.54,-3.54) -- (0,0);
               \draw (-0.35,0.3) arc(120:165:0.5)node[below=-4.7pt,right=-9.5pt]{$\theta$};
               \draw (0.35,-0.3) arc(120:165:-0.5)node[below=-4.7pt,right=-9.5pt]{$$};
                  
               \draw[dashed] (2,0)--(-1.2,0);
               \coordinate[label=above:$ \Gamma_*$] (A) at (-4.7,0.3);    
               \coordinate[label=above:$ \Gamma(t) $] (A) at (-4,2);  
               \coordinate[label=above: \small{$ \sigma   $}] (A) at (-2.2,0.6);         \end{tikzpicture}
        \caption{parameterizing of an evolving curve  over a fixed stationary curve}
        \label{Fi:parameter}
\end{figure}

 Let us   formulate the condition, that the curve $\Gamma(t)$ meets the $x$-axis at its boundary by 
\begin{equation} \label{contact and angle condition}
\langle \Phi(\sigma,t) , e_2 \rangle=0 \quad \text{ for } \sigma \in \partial \Gamma_*\,, \quad t\ge 0\,.
\end{equation}
Here and hereafter, $\langle \cdot , \cdot \rangle$ means the inner product in $\mR^2$. The following lemma shows that this condition leads to a linear dependency between $\mu$ and $\rho$ at the boundary points and as a result, nonlocal terms  will enter into  formulations. 

\begin{lem}\label{mu1}
Equivalent to the equation \eqref{contact and angle condition} is the following condition 
\begin{equation} \label{mu}
\mu = ( \cot \theta ) \rho  \quad \text{ on } \partial \Gamma_*\,.
\end{equation}
\end{lem}
\begin{proof}
Using the definition of $\Phi$, the fact that $\langle \sigma , e_2 \rangle = 0 $ on $\partial \Gamma_*$ and the  angle condition on  $\partial \Gamma_*$, we easily get 
\begin{equation*}
\mu=-\frac{ \langle N_* , e_2 \rangle}{ \langle n_{\partial \Gamma_*}, e_2 \rangle}\rho=-\Big(\frac{\cos \theta}{\cos( \frac{\pi}{2}+\theta)}\Big)\rho=(\cot \theta)\rho \quad \text{ on } \partial \Gamma_*\,
\end{equation*}
and vice versa.       
\end{proof}

We assume that the initial curve $\Gamma_0$ from \eqref{cap problem} is also given as  a 
 graph over $\Gamma_*$, i.e.,
$$
\Gamma_0 = \{ \Psi(\sigma, \rho_0(\sigma), \mu_0(\mathrm{pr}(\sigma))): \sigma \in \Gamma_*\}\,.
$$
Furthermore, In order to apply our main result we make the assumption that $\rho_0 \in C^{4+\a}(\Gamma_*)$ with $\|\rho_0\|_{C^{4+\a}}\leq \epsilon $ for some small $\epsilon>0.$  Note that since $\Gamma_0$ is assumed to satisfy the contact condition,  $\mu_0=(\cot \theta) \rho_0$ at $\partial \Gamma_*$.
\subsubsection{The    nonlocal, nonlinear parabolic boundary value problem} First we derive evolution equation for $\rho$ and $\mu$ which has to hold in the case that $\Gamma$ in \eqref{Gamma rho} solves  \eqref{cap problem}. Note that the following calculations are adapted from \cite{GarckeItoKohsakatriplejunction}. The normal velocity $V$ of $\Gamma (t)$ is given as 
\begin{equation*}
\begin{aligned}
V (\sigma,t) &= \langle  \Phi_t (\sigma, t),  N(\sigma, t) \rangle\\&= \frac{1}{J( \sigma, \rho ( \sigma,t) , \mu ( \mathrm{pr}(\sigma),t)}\langle \Psi_w , R \Psi_\sigma \rangle \rho_t( \sigma, t) + \langle \Psi_r , N(\sigma , t) \rangle \mu_t ( \mathrm{pr} (\sigma) , t ) \\
\end{aligned}
\end{equation*}
where 
\begin{align} \label{A normal}
N( \sigma ,t ) &= \frac{1}{J( \sigma, \rho ( \sigma,t) , \mu ( \mathrm{pr}(\sigma),t)} R \Phi_\sigma(\sigma ,t) \nonumber \\ 
               &= \frac{1}{J( \sigma, \rho ( \sigma,t) , \mu ( \mathrm{pr}(\sigma),t)} \big(R \Psi_\sigma + R \Psi_w \rho_ \sigma (\sigma ,t)\big) \,.
\end{align}

Here
\begin{equation} \label{J}
J=J(\sigma, \rho , \mu ) := \abs{ \Phi_ \sigma  } = \sqrt{\abs{\Psi_ \sigma}^2 + 2 \langle \Psi_ \sigma , \Psi_ w\rangle \rho_ \sigma  + \abs{\Psi_ w}^2 \abs{\rho_\sigma }^2} \,,
\end{equation}
and $R$ denotes the anti-clockwise rotation by $\pi/2$ (remember our convention  \eqref{x and sigma}).
In addition, the curvature $\kappa( = \kappa (\sigma, \rho, \mu))$ of $\Gamma(t)$ is computed as 
\begin{align} \label{A curvature}
     \kappa &= \frac{1}{(J(\sigma, \rho, \mu))^3} \langle \Phi_{\sigma \sigma} , R \Phi_{\sigma} \rangle \\
       &= \frac{1}{(J(\sigma,\rho, \mu))^3} \Big [ \langle \Psi _w, R \Psi_\sigma \rangle \rho_{\sigma \sigma} + \big\{2\langle \Psi_{\sigma w , }R \Psi_\sigma \rangle +  \langle \Psi_{\sigma \sigma}, R \Psi_w \rangle \big \}\rho_\sigma  \notag \\
       &\hspace{7pt}+ \{ \langle \Psi_{ww} , R \Psi_\sigma \rangle + 2 \langle \Psi_{\sigma w}, R \Psi_w \rangle + \langle \Psi_{ww}, R \Psi_w \rangle \rho_\sigma \} ( \rho_\sigma)^2 + \langle \Psi_{\sigma \sigma}, R \Psi_\sigma \rangle \Big ].\notag
 \end{align}

Thus the surface diffusion equation  can be formulated as 

\begin{equation}\label{equation for rho}
\rho_t = a(\sigma, \rho , \mu ) \Delta (\sigma, \rho , \mu ) \kappa (\sigma,\rho , \mu ) + b(\sigma, \rho , \mu ) \mu_t \,,
\end{equation}
where  
\begin{align*}
&a(\sigma, \rho , \mu ) := \frac{J(\sigma,\rho, \mu)}{\langle \Psi_w , R \Psi_\sigma \rangle}\,, \quad b(\sigma, \rho , \mu ) := - \frac{\langle \Psi_ r , R \Psi_\sigma\rangle + \langle \Psi_ r, R \Psi_w \rangle \rho_\sigma  }{\langle \Psi_w, R \Psi_\sigma \rangle}\,, \\
&\Delta(\sigma, \rho ,\mu) v:=  \frac{1}{J(\sigma, \rho, \mu)} \partial_\sigma \Big (  \frac{1}{J(\sigma, \rho, \mu)} \partial_ \sigma v \Big).  
\end{align*}
Note that we omitted the mapping $\mathrm{pr}$ in the function $\mu$  as well as the term   $ ( \sigma, \rho ( \sigma,t) , \mu ( \mathrm{pr}(\sigma),t)) $ in $ \Psi_u  $ with $u \in \{ \sigma, w, \mu  \}$ for reasons of shortness.

Now we will  write \eqref{equation for rho} as an evolution equation, which is nonlocal in space, just for the mapping $\rho$, using the linear dependence \eqref{mu} on $\partial \Gamma_*$. To do this, with the help of \eqref{mu}, we rewrite \eqref{equation for rho} into
\begin{equation}\label{suitable form of rho}
\partial_t \rho= \mathfrak{F}( \rho ,  \rho \circ \mathrm{pr}) +  \mathfrak{b}( \rho,   \rho \circ \mathrm{pr})\partial_t\left( (\cot \theta) \rho \circ \mathrm{pr}\right)  \quad \text{ in } \Gamma_*\,,
\end{equation}
where for $\sigma \in \Gamma_* \,$
\begin{align*}
\mathfrak{F}( \rho ,   \rho \circ \mathrm{pr}  ) (\sigma)&= a(\sigma, \rho , (\cot \theta)  \rho \circ \mathrm{pr})\Delta (\sigma, \rho , (\cot \theta)  \rho \circ \mathrm{pr} )\kappa (\sigma, \rho , (\cot \theta)  \rho \circ \mathrm{pr})\,,\\
\mathfrak{b}(\rho ,   \rho \circ \mathrm{pr}  )(\sigma) &= b(\sigma, \rho , (\cot \theta)  \rho \circ \mathrm{pr})\,.
\end{align*}  
By writing  \eqref{suitable form of rho} on $\partial \Gamma_*$ and rearranging it we are  led to 
\begin{equation*}
\big(1 - (\cot \theta) \mathfrak{b}( \rho , \rho \circ \mathrm{pr}  )\big)\partial_t \rho = \mathfrak{F}( \rho , \rho \circ \mathrm{pr}) \quad \text{ on }\partial \Gamma_* \,. 
\end{equation*}
Then, it follows that 
\begin{equation} \label{rho and projection}
 \partial_t \rho = \frac {\mathfrak{F}( \rho , \rho \circ \mathrm{pr}) }{1 - \mathfrak{(\cot \theta)b}( \rho , \rho \circ \mathrm{pr}  )}\quad \text{ on }\partial \Gamma_*\,. 
\end{equation}
Note that since $\rho \circ \mathrm{pr}= \rho$ on $\partial \Gamma_*$, \eqref{rho and projection} is purely an equation for  $\rho(\sigma)$ with $\sigma \in \partial \Gamma_*=\{ a_* , b_*\}$.
Near $\partial \Gamma_*$, where the projection $\mathrm{pr}$ is well-defined, the equation \eqref{rho and projection} leads to 
$$
 \partial_t \mu( \mathrm{pr} ( \sigma)) = ( \cot \theta )\partial_t  \rho( \mathrm{pr} (\sigma)) = ( \cot \theta )\Big \{  \frac {\mathfrak{F}( \rho , \rho \circ \mathrm{pr}) }{1 - \mathfrak{(\cot \theta)b}( \rho , \rho \circ \mathrm{pr})}\Big \}\circ \mathrm{pr } ( \sigma). \,
$$
Therefore the final equation for $\rho$ is  

\begin{equation} \label{final nonlinear problem}
\partial_t \rho= \mathfrak{F}( \rho , \rho \circ \mathrm{pr}) +  \mathfrak{(\cot \theta)b}( \rho , \rho \circ \mathrm{pr}  )\Big(\ \Big \{ \frac {\mathfrak{F}( \rho ,\rho \circ \mathrm{pr}) }{1 - \mathfrak{(\cot \theta)b}( \rho , \rho \circ \mathrm{pr}  )}\Big \}\circ \mathrm{pr } \Big ) \quad\text{ on } \Gamma_*.
\end{equation}
We emphasized that  the second term on the right hand side of this equation contains nonlocal terms including the highest order (i.e,  the fourth-order) point evaluation.

Furthermore, the boundary conditions on $\partial \Gamma_* = \{ a_* , b_* \}$ can be written as
\begin{align}
\mathfrak{G}_1 ( \rho) (\sigma) &: =  \langle N , e_2 \rangle - \cos \theta \nonumber\\ 
&\,{} =  \frac{1}{J( \sigma, \rho   , (\cot \theta)  \rho)} \big \langle R \Psi_\sigma + R \Psi_w \rho_ \sigma \,,\, e_2 \big\rangle - \cos \theta = 0 \,, \nonumber\\ 
\mathfrak G_2 (\rho)( \sigma ) &:=\partial_\sigma (\kappa (\sigma, \rho , (\cot \theta)  \rho)) = 0 \,. \label{final nonlinear bcs}
\end{align}
Note that the operators $\mathfrak{G}_1 $ and $\mathfrak{G}_2 $ are completely local as the projection $\mathrm{pr}$ acts as the identity on its image $\partial \Gamma_*$.
 
Altogether, by recalling the parameterization (see \eqref{x and sigma}),  we are led to the following nonlinear, nonlocal problem (see \cite[Equation (20)]{DepnerGarckeKohsaka} for the analogous result obtained for the  mean curvature flow):
\begin{equation}\label{nonlinear, nonlocal}
\left\{
\begin{aligned}
 \partial_t \rho(x,t) &= \mathcal F\Big(x ,\rho(x,t), \partial_x^1\rho(x,t) , \dots, \partial_x^4\rho(x,t),\dots \\
                      &\hspace{-20pt} \dots \rho(\pm l_*,t),\partial_x^1\rho(\pm l_{*},t),\dots, \boldsymbol{\partial_x^4 \rho(\pm l_{*},t)}\Big) & &\text{for } x \in [-l_{*},l_*],\\
0                     &=\mathcal G_1(x,\rho(x,t),\partial_x^1 \rho(x,t))
                                                                 & &\text{at } x=\pm l_*\,, \, \\
0                     &=\mathcal G_2(x,\rho(x,t),\partial_x^1 \rho(x,t),\dots,\partial_x^3 \rho(x,t))                                                       & &\text{at } x=\pm l_*\,, \, \\
\rho(x,0)             &=\rho_0(x)                                & &\text{for } x \in [-l_{*},l_*],\\ 
\end{aligned}
\right.
\end{equation}
where the term $\pm l_*$ should be understood in a sense that  $+l_*$ is taken in \eqref{nonlinear, nonlocal} for the values of $x$ in the neighborhood of $l_*$ and $-l_*$ is taken in \eqref{nonlinear, nonlocal} for the values of $x$ in the neighborhood of $-l_*$.

Note that the functions $\mathcal F, \mathcal G_1, \mathcal G_2$ are smooth with respect to  the $\rho$-dependent variables in some neighborhood of $\rho\equiv 0$ as well as  the first variable. Indeed as  you have seen above, these are rational functions with smooth coefficients  in the $\rho$-dependent variables  (possibly  inside of square roots which are equal to $1$ at $\rho\equiv0$, see \eqref{J}) with nonzero denominator at $\rho\equiv 0$.
 
\begin{rem}
Exactly at this point one needs to use the classical setting, e.g. the parabolic Hölder setting rather than the standard $L_p$-setting (which is  a natural choice), i.e., 
$$
W^{1,p}\big((0,T) ; L_{p}((-l_{*},l_*))\big) \cap L_p \big( (0,T) ; W^{4,p}((-l_{*}, l_*) \big ) \, 
$$
because of the nonlocal term $\boldsymbol{\partial_x^4 \rho(\pm l_{*},t)}$, see \eqref{nonlinear, nonlocal}, which can not be defined in this $L_p$-setting. 
\end{rem}

\subsection{Linearization and general setting}\label{linearization and abstract setting}
For the linearization  of \eqref{nonlinear, nonlocal} around $\rho\equiv 0 $, that is around the stationary solution $\Gamma_*$, we refer to \cite{DepnerGarckelinearized} (see also \cite{DepnerGarckeKohsaka}). More precisely,  the linearization of the surface diffusion equation is done in \cite[Lemma 3.2]{DepnerGarckelinearized} and a similar argument as in \cite[Lemma 3.4]{DepnerGarckelinearized} gives the following linearization of the angle condition
\begin{equation*}
\partial_{n_{\partial \Gamma_*}} \rho + \kappa_{n_{\partial \Gamma_*}}\mu=0 \quad \text{ on } \partial \Gamma_*\,.
\end{equation*}
Altogether, using the following facts (remind that $x$ is the arc-length parameter of $\Gamma_*$ and let $T_*$ denote the unit tangential vector of $\Gamma_*$) 
\begin{equation*}
\begin{aligned}
   \Delta_{\Gamma_*}\rho                 &=\partial_x^2 \rho    & & \text{for }x \in [-l_{*},l_*] \,, \\
    \partial_{n_{\partial \Gamma_*}}\rho &=\nabla_{\Gamma_*}\rho \cdot n_{\partial         \Gamma_*} = \partial_x \rho \,(T_* \cdot n_{\partial_{\Gamma_*}})=\pm         \partial_x \rho                                         & &\text{         at } x=\pm l_{*}\,, \\
    \kappa_{n_{\partial \Gamma_*}}            &=\kappa_*                  & & \text{         at }         x = \pm l_* \,, \\ 
    \mu                                  &=\cot \theta \rho     & & \text{         at } x = \pm l_* \,. 
\end{aligned}
\end{equation*}
we get for the linearization of  \eqref{nonlinear, nonlocal}  around $\rho\equiv0$ the following linear  equation for $\rho$ 
\begin{equation*}
\left\{
 \begin{aligned}
    \partial _t \rho + \partial_x^2( \partial_x^2 + \kappa_*^2) \rho &= f & &\text{         for } x\in [-l_{*},l_{*}]\,,  \\
    \pm\partial_x \rho +\kappa_{*}(\cot \theta)\rho                  &=g_{1} & &\text{         at }  x=\pm         l_{*}\,,  \\
    \partial_x(\partial_x^2 + \kappa_{*}^2 )\rho                     &=g_{2} & &\text{         at }  x=\pm         l_{*}\,.
 \end{aligned}
\right.
\end{equation*}

\begin{rem}
Note that the linearization does not have any nonlocal term particulary because of the fact that we linearized around stationary solutions. 
\end{rem}
Now the nonlinear, nonlocal problem \eqref{nonlinear, nonlocal} can   be restated as a perturbation of a linearized problem, that is of the form \eqref{eq1}, where $\Omega=(-l_*,l_*)$, the operator $A$ is given by 
$$
(Au)(x)=\partial_x^2( \partial_x^2 + \kappa_*^2)u(x)\,, \quad x\in[-l_{*},l_{*}]\,,
$$ 
 and the $B_j$'s are given by 
$$
(B_1u)(x)=\pm\partial_xu(x)+\kappa_*(\cot \theta) u(x)\,, \quad x=\pm l_*\,,
$$
$$
(B_2 u)(x)=\partial_x(\partial_x^2 + \kappa_*^2)u(x)\,, \quad x=\pm l_{*} \,.
$$
If we write \eqref{nonlinear, nonlocal} in the form of \eqref{eq1}, the corresponding $F$ is a regular function defined in a neighborhood of $0$ in $C^4(\o)$ with values in $C(\o)$. Indeed, it is  Frechet-differentiable of arbitrary order in a neighborhood of zero (using the differentiability of composition operators, see e.g. Theorem 1 and 2 of \cite[Section 5.5.3]{RunstSickel}) and a similar argument works for the corresponding functions $G_1$ and $G_2$. In particular,  the assumption (H1) is satisfied with $R = R'$ for sufficiently small $R'$. 

Clearly, the operators $A, B_1, B_2$ satisfy the assumption (H2), the operator $A$ is uniformly strongly parabolic and the operators $B=(B_1, B_2) $ satisfy the normality condition \eqref{normality condition}.

Let us verify that the linearized problem satisfies the complementarity condition, i.e., (L-S).
 For $ x=\pm l_*$ and  $ \lambda\in\overline{\mathbb{C}_+},\, \lambda\neq 0$ we should consider  the following ODE 
\begin{align}\label{L-S example}
\begin{cases}
 \lambda v(y)+\partial_y^4 v(y)=0\,, \quad y>0\,,\\
 \partial_y v(0)=0\,, \quad \partial_y^3 v(0)=0\,,
 \end{cases}
\end{align}
and prove that $v=0$ is the only solution which vanishes at infinity.  This can be done by the energy method. Testing the first line in \eqref{L-S example} with $\bar v$ and using the boundary conditions and the fact that $v$ and therefore its derivatives vanish at infinity (since  solutions of \eqref{L-S example} are the linear combinations of exponential functions) we obtain
\begin{align}
0&=\lambda\int_0^\infty \! |v|^2 \,\mathrm dy+\int_0^\infty \! \bar v \,\partial_y^4 v \,\mathrm dy\nonumber\\
&=\lambda\int_0^\infty \! |v|^2 \,\mathrm dy - \int_0^\infty \! \partial _y \bar v \,\partial_y^3 v \,\mathrm dy\nonumber\\
&=\lambda\int_0^\infty \! |v|^2 \,\mathrm dy+\int_0^\infty \! |\partial_y^2 v|^2 \,\mathrm dy\,. \nonumber
\end{align}
Since $0\neq\lambda\in\overline{\mathbb{C}_+}$, the function $v$ has to be zero and so the claim follows. 

Concerning the compatibility condition, as we have assumed that the initial curve satisfies the contact, angle, and no-flux conditions, we get at $x=\pm l_*$
\begin{eqnarray}
\left\{\begin{array}{lll}
\mathcal G_1(x,\rho_0(x,t),\partial_x^1 \rho_0(x,t))=0\,, \\
\mathcal G_2(x,\rho_0(x,t),\partial_x^1 \rho_0(x,t),\partial_x^2 \rho_0(x,t),\partial_x^3 \rho_0(x,t))=0 \,,\\
 \end{array}\right.
\label{compatibility example}
\end{eqnarray}
which is equivalent to the corresponding compatibility condition \eqref{compatibility} since we do not have zeroth-order boundary conditions.

\subsection{$\rho\equiv 0$ is normally stable}\label{check normally stability}
In this section, we will show that $\rho\equiv 0$, which  corresponds to $\Gamma_*$, is normally stable, i.e., it satisfies the assumption (i)-(iv) in Theorem \ref{theo01}. 

To begin with, let us consider the eigenvalue problem for the linearized operator $A_0$ (see \eqref{linear operator} for the precise definition of $A_0$) which reads as follows
\begin{equation}\label{eigenvalue}
\left \{
 \begin{aligned}
    \lambda u -\partial_x^2( \partial_x^2 + \kappa_{*}^2)u &=0 & &\text{ in         } [-l_*,l_*]\,,\\
    \pm\partial_x u+  \kappa_{*}\cot \theta\, u            &=0 & &\text{ at }          x=\pm l_*\,,\\
    \partial_x(\partial_x^2 + \kappa_{*}^2) u              &=0 & &\text{ at }          x=\pm l_*\,,
 \end{aligned}
\right.
\end{equation}
where $u \in D(A_0)$.
Multiplying  the first line in \eqref{eigenvalue} with $(\partial_x^2+\kappa_{*}^2)u$ and using integration by parts we get  
\begin{equation}\label{energy function}
-\lambda \, I(u,u)+\int_{-l_*}^{l_*}\!
(\partial_x(\partial_x^2 + \kappa_*^2) u)^2 \mathrm d x=0\,,
\end{equation}
where
$$
 I(u,u)=\int_{-l_*}^{l_*} (\partial_x u)^2 \mathrm d x-\kappa_*^2 \int_{-l_*}^{l_*} \! u^2\mathrm d x +  \kappa_*\cot \theta\left(u^2(l_*)+u^2(-l_*)\right)\,.
$$
Note that the same bilinear form appears in \cite[p. 1040]{GarckeItoKohsaka}  (taking $h_+=h_-=\kappa_*\cot \theta  $ in \cite{GarckeItoKohsaka}). Furthermore, we  refer to \cite[Proposition 3.3]{HutchingsMorganRitorAntonio}, where a related bilinear form appears as the second variation of the area functional for double bubbles. 

We first consider the case where $\lambda \neq0$.  The positivity of $I(u,u)$ is shown in \cite[Section 7]{GarckeItoKohsaka}, indeed we have 
$$
h=\kappa_*\cot \theta=\kappa_{*}\cot(-\kappa_* l_*)=-\frac{\kappa_*}{\tan (\kappa_* l_*)}\,,
$$
 which is the same equality as  in \cite[p. 1053]{GarckeItoKohsaka}.
 Now \eqref{energy function} implies  that all eigenvalues except zero are positive, in other word the operator $A_0$  satisfies the assumption (iv) in Theorem \ref{theo01}.

For $\lambda=0$, the bilinear form \eqref{energy function} implies  $\partial_x^2 u+ \kappa_{*}^2u= \tilde c $, where $\tilde c$ is a constant. It follows that  $u=a \sin (\kappa_*x)+ b \cos (\kappa_*x)+ c$, where $a$, $b$ and $c$ are constants. Applying the boundary conditions we get $b=-c \cos \theta\, $ and therefore we obtain a $2$-dimensional eigenspace for the eigenvalue $\lambda=0$. In fact we compute 
$$
N(A_0)=\mathrm{span}\,\{\sin(\kappa_* x)\,,\,1-(\cos \theta )\cos (\kappa_* x)\}\,.
$$ 
  
Next, let us  verify that  the eigenvalue $0$ of $A_0$ is semi-simple. Since the operator $A_0$ has a compact resolvent (see Remark \ref{isolated}),  the semi-simplicity condition is equivalent to the condition that $N(A_0)=N(A_0^2) $ (using  the spectral theory of compact operators, e.g. see \cite[Section 9.9]{Alt}).  To show this, it can  easily be seen that it is sufficient to prove the existence of a projection $$ P:X\rightarrow \mathcal R(P)=N(A_0)$$
such that $P$ commutes with $A_0$, that is, $PA_0 u =A_0Pu(=0)$ for all $u \in D(A_0)$. 

Indeed we can construct such a projection in the following way:
\begin{align}
P:X\rightarrow N(A_0):u\mapsto P u:= \alpha_1(u)v_1+ \alpha_2(u) v_2\,,
\end{align}
where
\begin{eqnarray*}
v_1=1-\cos \theta \cos(\kappa_*x)\,,\quad v_2=\sin(\kappa_* x)\,,
\end{eqnarray*}
\begin{eqnarray*}
\a_1(u)=\frac{\int^{l_*}_{-l_{*}} u(x)\,\mathrm d x}{\int^{l_*}_{-l_*} v_1(x)\mathrm d x}\,,\quad
\a_2(u)=\frac{(u-\a_{1}(u)v_1,v_2)_{-1}}{(v_2,v_2)_{-1}}\,.
\end{eqnarray*}
Here, the inner product is defined as 
$$
(\rho_1,\rho_2)_{-1}:= \int_{-l_*}^{l_*}\!\partial_x u_{\rho_1}\partial_x u_{\rho_2}\, \mathrm{d} x ,
$$
where $u_{\rho_i}\in H^1(-l_*,l_*)$ for a given  $\rho_i \in H^{-1}(-l_* , l_* ):=(H^1(-l_*,l_*))'$ with $\langle\rho_i,1\rangle_{H^{-1}, H^1}=0$ satisfies 

$$
\langle \rho_i, \varphi\rangle_{H^{-1},H^1}=\int_{-l_*}^{l_*}\!\partial_x u_{\rho_i} \partial_x \varphi \, \mathrm{d}x
$$
for all $\varphi \in H^1(-l_*,l_*)$ (see \cite[Section 4]{GarckeItoKohsaka} for more details). Here we denote by $\langle\cdot,\cdot \rangle_{H^{-1}, H^1}$ the duality paring between $H^{-1}(-l_*,l_*)$ and $H^1(-l_*,l_*)$.

Since 
$$
\int_{-l_*}^{l_*} v_1(x) \, \mathrm d x \neq0  , \quad \int_{-l_*}^{l_*} v_2(x) \, \mathrm d x \neq0 \quad \text {and } \int_{-l_*}^{l_*} u(x)-\a_1(u(x))v_1(x)\,\mathrm d x=0 \,,
$$ the coefficients $\a_1(u), \a_2(u)$
are well defined and moreover $\a_i(v_j)=\delta _{ij}$. Therefore  $P$ acts  as identity on its image $N(A_0)$ or equivalently we get $P^2=P$ and $R(P) = N (A_0)$.

Furthermore, for $u \in D(A_0)$ we have  
$$
\a_1(A_0 u)=\frac{\int^{l_*}_{-l_{*}} A_0 u(x)\,\mathrm d x}{\int^{l_*}_{-l_*} v_1(x) \,\mathrm d x} =\frac{\int^{l_*}_{-l_{*}} \partial_x^2( \partial_x^2 + \kappa_*^2)u\,\mathrm d x}{\int^{l_*}_{-l_*} v_1(x) \,\mathrm d x}=
\frac{ \partial_x( \partial_x^2 + \kappa_*^2)u|_{-l_*}^{l_*}}{\int^{l_*}_{-l_*} v_1(x) \, \mathrm d x}=0\,,
$$

$$
\a_2(A_0 u)=\frac{(A_0 u,v_2)_{-1}}{(v_2,v_2)_{-1}}=\frac{(u, A_0 v_2)_{-1}}
{(v_2,v_2)_{-1}}=0\,,
$$
where we have used the facts that $v_2 \in N(A_0)$ and the operator $A_0$ is symmetric with respect to the inner product $(\cdot,\cdot)_{-1}$ (see  \cite[Lemma 5.1]{GarckeItoKohsaka} ). Therefore 
 $$ 
 PA_0 u= \a_1(A_0u)v_{1}+\a_2(A_0u)v_2=0
 $$
which completes the proof of the existence of the desired projection. Consequently    the assumption (iii) in Theorem \ref{theo01} is verified.

We continue by  proving the assumption (i) in Theorem \ref{theo01}, i.e.,   near $\rho \equiv 0$, which corresponds to $\Gamma_*$, the set $\mathcal E$ of equilibria  of  \eqref{final nonlinear problem}, \eqref{final nonlinear bcs}  creates a $C^2$-manifold of dimension $2$.
 According to  \eqref{set of equilibria E}, $\rho \in \mathcal E $ if and only if
\begin{equation} \label{full equilibria} 
     \left \{
   \begin{aligned}
     0 &= \mathfrak{F}( \rho, \rho \circ \mathrm{pr}) \\
       &\qquad + \mathfrak{(\cot \theta)b}( \rho , \rho         \circ \mathrm{pr}  )\Big(\ \Big \{ \frac {\mathfrak{F}( \rho ,\rho         \circ \mathrm{pr}) }{1 - \mathfrak{(\cot \theta)b}( \rho , \rho \circ \mathrm{pr}  )}\Big \}\circ \mathrm{pr } \Big )                 & &\text{on } \Gamma_* \,, \\
       0& = \mathfrak{G}_1 ( \rho ) & &\text{on } \partial \Gamma_* \,, \\
       0& =  \mathfrak G_2 ( \rho ) & &\text{on } \partial \Gamma_* \,.
  \end{aligned}
\right.
\end{equation}
Here and in what follows  we omit the condition $\rho \in B_{X_1}(0,R)$ from the right hand side  for reasons of shortness.
Similarly as before, by writing the first line in \eqref{full equilibria} on $\partial \Gamma_*$ we get  $\mathfrak{F}( \rho, \rho \circ \mathrm{pr}) = 0$ on $\partial \Gamma_*$ and hence   
\begin{equation*}
   \rho \in \mathcal E \Leftrightarrow  
   \left \{
   \begin{aligned}
       &0= \mathfrak{F}( \rho, \rho \circ \mathrm{pr}) & &\text {on } \Gamma_* \,, \\
             &0 = \mathfrak{G}_1 ( \rho ) & &\text{on } \partial \Gamma_* \,, \\
       &0 =  \mathfrak G_2 ( \rho ) & &\text{on } \partial \Gamma_* \,.
  \end{aligned}
\right.
\end{equation*}
Using the definition of $\mathfrak F$ and no-flux condition $\mathfrak G_2$, by applying Gauss's theorem it follows that   
\begin{equation*}
   \rho \in \mathcal E \Leftrightarrow
   \left \{
   \begin{aligned}
      &\rho \in B_{X_1}(0,R)\,, \\
      &\kappa \big( \rho, ( \cot \theta )\rho \circ \mathrm{pr}\big) \text{ is constant } & &\text {on } \Gamma_* \,, \\
      &\mathfrak{G}_1 ( \rho )          = \langle N , e_2 \rangle - \cos \theta  = 0 & &\text{on } \partial \Gamma_* \,. \\
  \end{aligned}
\right.
\end{equation*}
    Therefore, by taking into account  Lemma \ref{mu1}  we conclude that      
$$
\mathcal E=\Big\{\rho: \rho \text{ parameterizes an element of } CA_r(a_1,-r\cos \theta) \text{ sufficiently close to $\Gamma_*$} \Big\}.
$$
Clearly $\mathcal E \neq \emptyset$ as $\rho \equiv 0$ parameterizes $\Gamma_* = CA_{r_*}(0,-r_* \cos \theta)$.  The following lemma demonstrates that actually,   all the circular arcs $CA_r(a_1,-r\cos \theta)$ sufficiently close to $\Gamma_*$ can be parameterized by a unique function $\rho$    depending smoothly on  $a_1$ and $r$. The idea is to use the  implicit function theorem of Hildebrandt and Graves, see Zeidler \cite[Theorem 4.B]{Zeidlernonlinear1}.
\begin{lem} \label{Lem: manifold}
There exist positive numbers $\epsilon$ and $R''$ such that each of the circular arcs  $CA_{r}(a_{1}, -r \cos \theta)$ with $(a_1, r ) \in B_{\mR^2} ((0, r_* ), \epsilon )$  is parameterized by a unique   $\rho \in B_{X_1}(0, R'')$. Moreover the set $\mathcal E$ creates a $C^2$-manifold of dimension $2$ in $X_1 = C^{ 4 + \a}([-l_*, l_*] )$.
\end{lem}
\begin{proof}Without loss of generality we may assume that $\Gamma_*$ is centered at the origin of $\mathbb R^2$. We use the implicit function theorem with 
\begin{eqnarray*}
X = \mR^2 \,, & Y = Z = C^{4+\a}([-l_* , l_*])\,, &(x_0,y_0) =((0, r_{*}),0)
\end{eqnarray*}
and 
\begin{equation*} 
 F : X \times Y  \longrightarrow  Z  
\end{equation*}
where
(remember our abuse of notation \eqref{abuse of notation}) \begin{equation} \label{Eq: implicit}
F( a_1 ,r ,\rho )(\sigma):=\norm{\Psi(\sigma,\rho ( \sigma ),\mu(\mathrm{pr}(\sigma )) - (a_1, r \cos \theta )}^2 - r^2 \,,
\end{equation}
for all $(a_1,r) \in X,\, \rho \in Y$ and $\sigma \in [ - l_* , l_* ]$. Here 
$$\Psi (\sigma, \rho ( \sigma ), \mu(\mathrm{pr} ( \sigma ) ) = \sigma + \rho(\sigma) N_*(\sigma) + \mu ( \mathrm{pr} ( \sigma )) \tau_*(\sigma)  , \quad  \mu \circ \mathrm{pr} = (\cot \theta )\rho \circ \mathrm{pr}\,.
$$ 
The derivative $F_\rho ( 0, r_*, 0)$ is given by 
$$
F_{\rho}( 0, r_*, 0)(v)(\sigma) \ = \big \langle vN_*(\sigma) + (\cot \theta) (v \circ \mathrm{pr)} \tau_*(\sigma)\, , \, \sigma - ( 0, -r_* \cot \theta) \big \rangle \,.   
$$
Using the fact that $
\sigma-( 0, -r_* \cot \theta)=r_* N_*(\sigma)
$  (see Figure \ref{fig:}) and that $\tau_*$ is a tangential vector field, we get
$$
F_{\rho}( 0, r_*, 0)(v)= r_* v 
$$
which implies that $F_\rho ( 0, r_*, 0)$ is bijective. Furthermore, it is easy to see that $F$ is a smooth map on a neighborhood of $ (0, r_*, 0)$. 

Hence  there exist positive numbers $\epsilon$ and $R''$ such that, for every $(a_1, r ) \in B_{\mR^2} ((0, r_* ), \epsilon )$, there is  exactly one    $\rho(a_{1}, r) \in X_1$ for which $\rho \in B_{X_1}(0, R'')$ and $F(a_{1}, r, \rho(a_{1},r))=0$, i.e., 
\begin{align} \label{Eq: zero level set}
\norm{\Psi\big(\sigma,\rho ( \sigma, a_1, r),\mu(\mathrm{pr}(\sigma ))\big) - (a_1, r \cos \theta )}^2 - r^2 = 0 \quad \text{ for } \sigma \in \Gamma_* \,.
\end{align}
 In addition the mapping $(a_{1}, r) \mapsto \rho ( a_1, r )$ is smooth on a neighborhood of $ x_0 = (0,r_*) $. Finally it is not  hard to see that the curve $\Gamma$ parameterized by $\rho = \rho (a_1, r)$ (i.e., the solution to $F=0$) belongs to $CA_r ( a_1, -r \cos \theta )$. Indeed, the contact condition is satisfied as we have already included here the linear dependency  \eqref{mu} and now  taking into account the  relationship between the center and the radius (see \eqref{Eq: implicit}),  we find easily that  the curve $\Gamma$ satisfies   the desired angle condition (see Figure \ref{fig:spherical arc}). This  proves the first assertion  of  the lemma.

Define a function
\begin{align*}
\Upsilon: U        &\longrightarrow X_1 \\
      ( a_1, r )&\mapsto \rho ( a_1 , r )\,,  
\end{align*}
where $U = B_{\mR^2}((0, r_*), \epsilon )$.
Clearly, the function $\Upsilon$ is smooth and so  in particular $C^2$. Furthermore, $\Upsilon ( U ) = \mathcal E $ with the constant $R$ in  defining relation \eqref{set of equilibria E}  replaced by $R''$; and     $ \Upsilon ( (0,r_*) ) = 0 $. Now to prove that  the set $\mathcal E$ creates a $C^2$-manifold of dimension $2$ in $X_1 = C^{ 4 + \a}([-l_*, l_*] )$ we only need to verify  that the rank of $\Upsilon'((0,r_*)) $ is equal to $2$. (See the definition of a manifold on page  \pageref{def equilibria1}.)  

Differentiating \eqref{Eq: zero level set} with respect to $r$ and evaluating it at $( a_{1}, r) = ( 0, r_* )$,  we get 
\begin{align*}
\big\langle \partial_r \rho(\sigma, 0, r_*)N_*(\sigma)+\partial_r \mu(\mathrm{pr}(\sigma), 0, r_*)\tau_*(\sigma)- (0, &-\cos \theta) \,,\\
& \sigma- (0, -r_* \cos \theta)\big \rangle - r_* = 0.
\end{align*}
Again using the fact that $
\sigma- (0, -r_* \cos \theta) =  r_* N_*(\sigma)
$   and that $\tau_*$ is a tangential vector field, we get 
$$
r_* \partial _r \rho(\sigma, 0, r_*)+ \cos \theta (\sigma_2+r_* \cos \theta)=r_*\,.
$$
By writing it in spherical coordinates, i.e., $$ \sigma=(\sigma_1,\sigma_2)=\Phi_*(x)=\left(r_*\sin(\frac{x}{r_*})\,,\, r_*\cos(\frac{x}{r_*}) -r_* \cos \theta\right)$$ we obtain 
$$
\partial_r \rho( x, 0, r_* )=1- \cos \theta \cos (\kappa_* x)\,.
$$ 
Analogously, we get $\partial_{a_1} \rho( x, 0, r_* )= -\sin(\kappa_* x)$, which finishes the proof.

\end{proof}
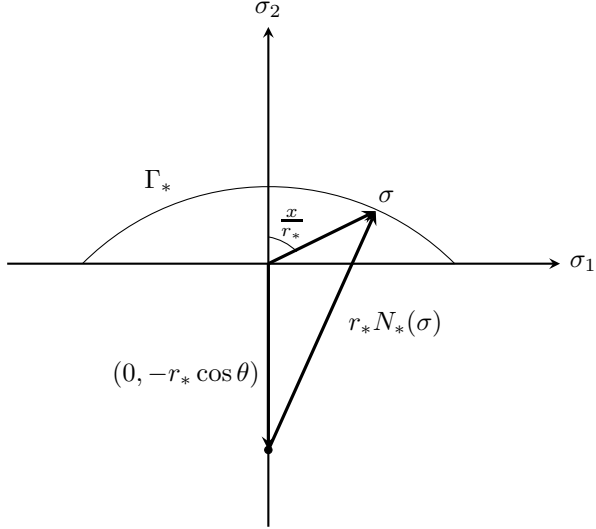
\begin{figure}[htbp]
        \centering
        \begin{tikzpicture}[scale=0.7,>=stealth]
                \draw (0,0) arc (45:135:5)node[below=-1.1cm,right=0.7cm]{$\Gamma_*$};
               \draw [->,thick](-8.5,0)--(2,0)node[right]{$\sigma_1$};
               \draw [->,thick](-3.54,-5)--(-3.54,4.5)node[above]{$\sigma_{2}$};
               \filldraw [black] (-3.54,-3.54) circle (2pt)node[below=-1cm, right=-2.2cm]{$(0,-r_*\cos \theta)$};
                \draw [->,very thick](-3.54,0) -- (-3.54,-3.54);
                \draw [->,very thick](-3.54,0) -- (-1.5,1)node[below=-0.2cm,right=-.1cm]{$\sigma$};
               \draw[->,very thick] (-3.54,-3.54) -- (-1.5,1)node[below=1.5cm,right=-0.5cm]{$r_* N_*(\sigma)$};
               \draw (-3.54,0)--(0,0);
               \draw (-3.,.23) arc(45:80:1)node[below=-5pt,right=0pt]{$\frac{x}{r_*}$};      
                      \end{tikzpicture}
                       \caption{The stationary solution $\Gamma_*$}
               \label{fig:}
\end{figure}
Finally it remains to prove the assumption (ii). This is  an immediate consequence of the facts that  $T_{0} \mathcal E \subseteq N(A_0) $, see \eqref{inclusion N}, and that $\mathrm{dim}( \mathcal E) = \mathrm{dim}( N( A_0 ))$. 

In summary, all the assumptions of Theorem \ref{theo01} for $R=\mathrm{min}\{R',R''\}$ are satisfied. Thus applying  Theorem \ref{theo01}, we obtain 
\begin{theo}\label{main theorem example}
Suppose $\Gamma_*$ is an arbitrary circular arc intersecting the $x^1$-axis with an  angle $\theta$. Then $\rho\equiv 0$ is a stable equilibrium of \eqref{nonlinear, nonlocal} in the class of all initial values $\rho_0 \in X_1=C^{4+\a}([-l_*,l_*])$ satisfies   the compatibility condition \eqref{compatibility example}. Moreover there exists a $\delta>0$ such that if $\norm{\rho_0}_{X_1}< \delta$ then the corresponding solution of \eqref{nonlinear, nonlocal} exists globally in $ C^{1+\frac{\alpha}{4}, 4+\alpha}([0,\infty)\times[-l_*,l_*])$ and converges at an exponential rate in $X_1$ to some equilibrium $\rho_\infty$ as $t\rightarrow \infty$. 
\end{theo}
In this sense, the lens-shaped network generated by $\Gamma_*$ is  stable under the surface diffusion flow. In addition, every lens-shaped solution of \eqref{surface diffusion for lens}  that starts sufficiently close to the one generated by $\Gamma_*$ and satisfies the angle condition  and the balance of flux condition at $t=0$ exists globally and converges to some lens-shaped network generated by a circular arc at an exponential rate as $t\rightarrow \infty$.

%% file: append.tex
\section{Appendix}
Throughout the appendix we follow the notation of the previous sections, except that here $u'$ denotes the derivative of a function $u$ with respect to time.

Let $\sigma^-(-A_0)$ denote the subset of $\sigma(-A_0)$ consisting of elements with negative real parts. Note that  $\sigma^-(-A_0)$ is a spectral set due to  Remark \ref{isolated}. Clearly $\sigma^-(-A_0) = -\sigma_s $ and $P^- = P^s$, where $P^-$ is the spectral projection associated to $\sigma^-(-A_0)$.

\subsection{Asymptotic behavior for linear scalar equations}\label{abilp}
Such a result is proven  in \cite[Theorem.~0.1]{BraunerHulshofLunardi} for a single equation of second order  with first-order boundary condition. Here we  extend this result to a single equation of order $2m$ with $m$ boundary conditions.  Precisely we consider the  linear problem \eqref{linear} 
with $\boldsymbol {N=1}$, i.e., 
\begin{equation} \label{linear single}
\left\{
 \begin{aligned}
    \partial _t u+Au &= f(t)  & &\mbox{ in }\Omega \,,         & &t \geq0         \,, \\ 
     Bu              &= g(t)  & &\mbox{ on }\partial\Omega \,, & &t \geq0         \,,\\
     u(0)            &= u_{0} & &\mbox{ in }\Omega \,, \\
 \end{aligned}
\right.
\end{equation}
where $\Omega$ is a bounded domain in $\mathbb{R}^n$ with $C^{2m+\a}$ boundary, $0<\a<1$, $g=(g_1,\dots,g_{m}), B=(B_1, \dots, B_{m})$, $u_0 \in C^{2m+\a}(\o)$ and the operators $A$ and $B$ satisfy the  conditions (H2), (L-S), (SP),  and  the normality condition  \eqref{normality condition}. Note that the normality condition   in particular implies that 
$$
0 \leq m_1 < m_2<\cdots<m_m \leq 2m-1 \,.
$$

For  convenience, we  set 
$$
\mathcal{L}=-A \quad \text{ and }\quad L=-A_0\,.
$$
The realisation $L$ of $\mathcal L$ with homogeneous boundary conditions in $X=C(\o)$,  defined similarly as \eqref{linear operator}, is a sectorial operator by Theorem \ref{known result}. Furthermore,  if $f \in \mathbb E_0(T)$, $g \in \mathbb F(T)$ and $u_0 \in C^{2m+ \a}(\o)$ satisfying the compatibility condition \eqref{compatibility linear}, the unique solution of \eqref{linear single} belongs to $\mE_1(T)$ for all $T$ and in addition it  is given by the extension of
the Balakrishnan formula with some adaptations (see (37)-(40) of §7 in \cite{lunardisinestrariwah})
\begin{align}
 u(\cdot,t)&= e^{tL}(u_0-n(\cdot,0))+\int_0^t \! e^{(t-s)L}[f(\cdot,s)+\mathcal{L}
 n(\cdot,s)- n'_1(\cdot,0)]\,\mathrm{d}s\nonumber\\
 & \quad +n_1(\cdot,t)-\int_0^t \! e^{(t-s)L}(n_1'(\cdot,s)-n_1'(\cdot,0))\,\mathrm{d} s \nonumber\\
 & \quad+n_2(\cdot,0)-L \int_0^t \!e^{(t-s)L}[n_{2}(\cdot,s)- n_2(\cdot,0)] \, \mathrm{d} s \label{BF 1}\\
 &= e^{tL} u_0 + \int_0^t \! e^{(t-s)L}[f(\cdot,s)+\mathcal{L}n(\cdot,s)]
\, \mathrm{d} s  \nonumber\\
& \quad -L\int_0^t \! e^{(t-s)L}n(\cdot,s)\,\mathrm d s\,, \quad 0 \leq t \leq T \,.  \label{Balakrishnan formula} 
\end{align}
Here 
\begin{equation*}
n(t)= \mathcal{N}( g_1(t), \dots, g_m(t))=\sum\limits_{s=1}^m \mathcal{N}_s \mathcal{M}_s ( g_1(t), \dots, g_s(t))\,,  
\end{equation*}
\begin{equation}
n_1(t)=\begin{cases}
         0& \text{ if } m_1>0\\
         \mathcal{N}_1 \mathcal{M}_1(g_1(t)) & \text{ if } m_1=0\,
       \end{cases} \quad \text{ and } \quad n_2(t)=n(t)-n_1(t)\,,
              \end{equation}
 where the operator $\mathcal{N}$ given in the following theorem is a lifting operator with an explicit construction such that
\begin{align}\label{extension operator for N=1}
\begin{cases}
\mathcal{N}\in L(\prod_{j=1}^m C^{2m+\theta'-m_j}(\partial \Omega),C^{2m+\theta'}(\o)),\quad \forall \, \theta'\in [0,\a]\,,\vspace{8pt}\\
B_j(\mathcal{N}(g_1,\dots, g_m))(x)=g_j(x), \quad x\in \partial \Omega ,\quad j=1,\dots, m\,.
\end{cases}
\end{align} 
\begin{theo}\label{scalar result theorem}
Given $s=1,\dots, m$, there exist
\begin{equation*}
\mathcal{M}_s\in L(\prod_{j=1}^s C^{\theta-m_j}(\partial \Omega),C^{\theta -m_s}(\partial \Omega)),\quad \forall \theta\in[m_s,2m+\a],
\end{equation*}
and
\begin{equation*}
\mathcal{N}_s\in L(C^r(\partial \Omega); C^{r+m_s}(\o)),\quad \forall r\in [0,2m+\a -m_j]
\end{equation*}
such that, setting 
\begin{equation*}
\mathcal{N}(\psi_1,\dots, \psi_m)= \sum_{s=1}^m \mathcal{N}_s \mathcal{M}_s(\psi_1,\dots,\psi_s)\,, \end{equation*}
we have 
\begin{equation} \label{mathcal N1}
\mathcal{N}\in L(\prod_{j=1}^m C^{2m+\theta'-m_j}(\partial \Omega),C^{2m+\theta'}(\o)),\quad \forall \theta'\in [0,\a]\,,
\end{equation}
and
\begin{equation*}
B_j(\mathcal{N}(\psi_1,\dots, \psi_m))(x)=\psi_j(x), \quad x\in \partial \Omega ,\quad j=1,\dots, m\,.
\end{equation*}
Moreover, for each $u \in C(\partial \Omega)$,
\begin{equation}\label{N_s}
D_x^l \mathcal N_s u(x)=0 \,, \quad x \in \partial \Omega \,,\, l \in \mathbb{N}^n \,, \, |l|<m_s \,, 
\end{equation}
which in particular implies that
\begin{equation*}
(B_j \mathcal{N}_s u)(x) \equiv 0,  \quad x \in \partial \Omega, \text{ for } j<s\,.
\end{equation*}
\end{theo}
\begin{proof}
The proof is given in  \cite[Theorem~6.3]{lunardisinestrariwah}.
\end{proof}
\begin{theo}\label{asymptotic behavior theorem}
Let $0< \omega <-\mathrm{max}\{ \, \mathrm{Re} \,\lambda :\lambda \in \sigma^-(-A_0)\}$. Suppose $f$ and $g$ are such that $(\sigma,t)\mapsto e^{\omega t}f(\sigma, t)\in \mE_0(\infty)$ and  $(\sigma , t)\mapsto e^{\omega t}g(\sigma, t) \in \mF(\infty)$. Suppose further that $u_0 \in C^{2m+\a}(\o)$ satisfy the compatibility condition \eqref{compatibility linear}. Let  $  u$ be  the solution of \eqref{linear single}.
Then $v(\sigma , t):= e^{\omega t}u(\sigma, t)$ is bounded in $[0,+\infty)\times \o $ if and only if 
\begin{align}\label{condition on u_0}
(I-P^-)u_0=&-\int_0^{+\infty} \! e^{-sL}(I-P^-)[f(\cdot,s)+\mathcal{L}\mathcal{N}
g(\cdot,s)]\,\mathrm{d}s \nonumber\\
           &+L\int_0^ \infty \! e^{-sL} (I-P^-)\mathcal{N}g(\cdot, s)\, \mathrm{d}s\,. \end{align}
If this is so, the function $u$ is given by 
\begin{align}
u(\cdot,t) &= e^{tL} P^- u_0+ \int_0^t \! e^ {(t-s)L} P^-[f(\cdot, s)+\mathcal{L}\mathcal{N}g(\cdot,s)] \, \mathrm{d} s \nonumber \\ 
           &\quad -L \int_0^t \! e^{(t-s)L}P^-\mathcal{N}g(\cdot,s)\, \mathrm{d} s \nonumber\\ 
           &\quad -\int_t^{+ \infty }\! e^{(t-s)L}(I-P^-)[f(\cdot,s)+\mathcal{L} \mathcal{N} g(\cdot,s)]\, \mathrm{d}s  \nonumber\\
           &\quad +L \int_t^{+\infty} \! e^{(t-s)L}(I-P^-)\mathcal{N}g(\cdot,s)\,                   \mathrm{d} s,\label{expression for u}        
\end{align}
and the function $v=e^{\omega t}u$ belongs to $\mE_1(\infty)$, with the estimate
\begin{align} \label{estimate in terms of the datas}
\|& v \|_{\mE_1(\infty)} \leq C( \| u_0 \|_{C^{2m+\a }(\o)} + \| e^{ \omega t} f \|_{\mE_0(\infty)}+\|e^{ \omega t}g\|_{\mF(\infty)})   
\end{align}
for some $c > 0$ independent of $(u_0, f, g)$.
\end{theo}
\begin{proof}
The proof follows the arguments of \cite[Theorem~0.1]{BraunerHulshofLunardi}. The novelty with respect to \cite{BraunerHulshofLunardi} is the appearance of systems of $m$ boundary conditions (including possibly zeroth-order boundary conditions) which  is treated   with the method introduced in \cite[Section~7]{lunardisinestrariwah}. 

Taking into account the estimates (see \cite[Proposition 2.3.3]{Lunardi1995424}) (which hold for small $\epsilon > 0$ and for $t>0$)
\begin{align*}
\|P^- e^{tL}\|_{L(X)}      & \leq Ce^{-(\omega +\epsilon)t}\,,\\
\|LP^- e^{tL}\|_{L(X)}     & \leq \frac{C e^{-(\omega +\epsilon)t}}{t} \,,\\
\|e^{-tL}(I-P^-) \|_{L(X)} & \leq Ce^{- (\omega-\epsilon)t}\,,
\end{align*} 
and arguing as in \cite{Lunardi1995424}, one can easily verify that the function given by the right hand side of \eqref{expression for u} is bounded by $Ce^{-\omega t}$.

In view of  \eqref{Balakrishnan formula}, we have $u=u_1+u_2$, where $u_1$ is the function on the right hand side of \eqref{expression for u} and
\begin{align*}
u_2(\cdot,t)&= e^{tL}\Bigg( (I-P^-)u_0 + \int_0^ \infty \! e^{-sL}(I-P^-)(f(\cdot, s)+\mathcal{L}\mathcal{N}
g(\cdot,s))\,\mathrm{d} s \\
& \qquad-L\int_0^\infty \! e^{-sL}(I-P^-)\mathcal{N}
g(\cdot,s)\,\mathrm{d}s \Bigg)\\
& \equiv e^{tL}y\,, \quad t\geq 0\,.
\end{align*}
From our assumption on $\omega$,  it follows that 
\begin{equation} \label{choice of  omega}
\sigma(L + \omega I )  \cap i\mR = \emptyset
\end{equation}
 and the projection $(I-P^{-})$ is the spectral projection associated to the unstable part of $\sigma(L+\omega I)$. Therefore due to     $$e^{\omega t}u_2(\cdot,t)=e^{t(L+\omega I)}y\,$$
and the fact that $ y$ is an element of $(I-P^-)(X),$     $e^{\omega t}u_2(\cdot,t)$ is bounded in $[0,\infty)$ with values in $X$ (i.e., $v$ is bounded) if and only if $y=0,$  i.e., iff \eqref{condition on u_0} holds.

We now prove that $v=e^{\omega t}u \in \mE_1(\infty)$. First note that $v$ solves \eqref{linear single}  with $\mathcal{L}$ replacing $\tilde {\mathcal{L}}=\mathcal{L} +\omega I$, and $f$ and $g$ replacing $\widetilde f=f e^{\omega t}$ and $\widetilde g=g e^{\omega t}$ respectively. Due to  the regularity of the data and the  compatibility condition \eqref{compatibility linear}, by Proposition \ref{M-R}, $v$ belongs to  $\mE_1(1)=C^{2m+\a,1+\frac{\a}{2m}}(\o \times [0,1]) $ and 
 $$
 \norm{v}_{\mE_1(a)}\le C \big(\abs{u_0}_1+\|\widetilde f \|_{\mE_0(\infty)}+\| \widetilde g \|_{\mF(\infty)})\,.
$$   
Hence it remains to show that $v \in C^{2m+\a, 1+\frac{\a}{2m}}(\o \times [1,\infty))$. 
As a result of   \eqref{choice of omega}, we have the following estimates for some $\gamma>0$:
\begin{equation} \label{projection estimates}
\begin{aligned}
\| \widetilde L^k e^{t\widetilde L}P^- \|_{L(X)} &\leq \frac{C_k e^{-\gamma         t}}{t^k}\,,  & &t > 0 \,, \\
\|\widetilde L^k  e^{-t\widetilde L}(I-P^-) \|_{L(X)}&\leq C_k e^{-\gamma         t},          & &t > 0 \,, \quad k\in \mathbb{N}\,.
\end{aligned}
\end{equation}
Let us define 
\begin{equation}
\widetilde n(t):= \mathcal{N}(\widetilde g_1(t), \dots, \widetilde g_m(t))=\sum\limits_{s=1}^m \mathcal{N}_s \mathcal{M}_s (\widetilde g_1(t), \dots, \widetilde g_s(t))  \end{equation}
 and
\begin{equation}
   \widetilde n_1(t):=
        \begin{cases}
           0& \text{ if } m_1>0\\
           \mathcal{N}_1 \mathcal{M}_1(\widetilde g_1) & \text{ if } m_1=0
        \end{cases} 
   \quad \text{ and } \quad \widetilde n_2(t):=\widetilde n(t)-\widetilde         n_1(t)\,.
\end{equation}
By decomposing $v$ as $v=P^- v+(I-P^-)v$, using the  equality \eqref{BF 1} for the term $P^- v$, the equality \eqref{Balakrishnan formula} for the term $(I-P^- )v$ and taking into account \eqref{condition on u_0}, we can split $v(t)=v(\cdot,t)$ as $v=\sum_{i=1}^5 v_i$, where
\begin{align*}
v_1(t) &=  e^{t \widetilde L }P^-( u_0-\widetilde n(0))+ \int_0^t \! e^{(t-s)\widetilde{L}}            P^- [\widetilde f (s)+ \widetilde {\mathcal{L}} \widetilde n(s)-\widetilde n_1'(0)]            \,\mathrm{d}s\,,\\
v_2(t) &=  P^- \widetilde n_1(t) - \int_0^t \! e^{(t-s)\widetilde{L}}P^-(\widetilde n_1'(s)-\widetilde n_1'(0))            \,\mathrm{d} s\,,\\
v_3(t) &=  P^- \widetilde n_2(0)-\widetilde{L} \int_0^t \! e^{(t-s)\widetilde{L}}
           P^- (\widetilde n_2(s)-\widetilde n_2(0))\, \mathrm{d}s\,,\\ 
v_4(t) &=  -\int_t^\infty \! e^{(t-s)\widetilde{L}}(I-P^-)[\widetilde f(s) +            \widetilde{\mathcal L} \widetilde n(s)]\,\mathrm{d} s \,,\\
v_5(t) &=  \widetilde L \int_t^\infty \! e^{(t-s)\widetilde L } (I-P^-)
\widetilde n(s)            \,\mathrm{d} s\,.                       
\end{align*}

Furthermore, we need the following facts about the regularity of  $\widetilde n$, which are proven in \cite{lunardisinestrariwah}, see (5),(11)-(13) of §7 in this paper:
\begin{align*}
\begin{cases}
\widetilde n &\in B([0,\infty);\, C^{2m+\a}(\o))\cap C^{\frac{\a}{2m}}([0,\infty);\,C^{2m}(\o))\,,\\
\widetilde{\mathcal L}\widetilde n &\in B([0, \infty);\, C^{\a}(\o)) \cap C^{\frac{\a}{2m}}([0,\infty);X)\,,\\
\widetilde n_1 &\in B([0,\infty);\, C^{2m+\a}(\o)) \,, \\
\widetilde{\mathcal L}\widetilde n_1 &\in B([0, \infty);\, C^{\a}(\o)) \,,\\
\widetilde n_1' &\in C^{\frac{\a}{2m}}([0,\infty);X)\cap B([0, \infty); C^{\a}(\o))\,.
\end{cases}
\end{align*}

Let us first consider $v_1$. Since $t \rightarrow \widetilde f (\cdot, t)$, $t \rightarrow \widetilde{\mathcal L } \widetilde n(s)$ and $\widetilde n_1'(0)$ belong to 
$C^{\frac{\a}{2m}}([0,\infty);X)$ 
by \cite[Proposition 4.4.1(ii)]{Lunardi1995424}, we have 
\begin{align} \label{first v_1}
\begin{cases}
v_1\in C^{1+\frac{\a}{2m}}([1,\infty);X),\vspace{5pt}\\
v_1(t) \in D( A_0) \subseteq \bigcap\limits_{p>1} W^{2m,p}(\Omega) \qquad t \in [1, \infty) \,, \\
v_1' \in B([1,\infty);C^{\a}(\o)),
\end{cases}
\end{align}
and 
\begin{align}\label{v_1}
\begin{cases}
v_1'(t)=\widetilde{L} v_1(t)+ P^-[\widetilde f(t)+ \widetilde{\mathcal L } \widetilde n(t)-\widetilde n_1'(0)]\,,\quad t\geq 0\,,\\
v_1(0)= P^-(u_0-\widetilde n(0))\,,
 \end{cases}
\end{align}
where we have used the fact that $D_{\widetilde L}(\frac{\a}{2m},\infty)\simeq C^\a(\o)$ (by Theorem \ref{known result} (ii)).
 On the other hand, since $\widetilde f$, $\widetilde{\mathcal L } \widetilde n$, $\widetilde n_1'(0)$ and $v_1'$ belong to $B([1, \infty); C^{\a}(\o))$, by \eqref{v_1} we conclude that $\widetilde{\mathcal L }v_1 \in B([1, \infty); C^{\a}(\o))$.

Summing up we obtain 
\begin{align}\label{property}
\begin{cases}
v_1\in C^{1+\frac{\a}{2m}}([1,\infty);X),\qquad v_1' \in B([1,\infty);C^{\a}(\o)),\vspace{10pt}\\
 \widetilde{\mathcal L} v_{1} \in B([1,\infty);C^\a(\o)), \qquad v_{1}(t) \in \bigcap\limits_{p>1} W^{2m,p}(\Omega), \quad t \in [1,\infty)\,.
\end{cases}
\end{align}

Considering $v_2$,   since $ n_1'(t)-n_1'(0) \in C^{\frac{\a}{2m}}([0,\infty);X)$, by \cite[Proposition 4.4.1(ii)]{Lunardi1995424} $v_2$ satisfies the same properties as $v_1$ stated in \eqref{first v_1}  and $v_2(t)=P^- \widetilde n_1(t)+y(t)$,
where $ y(t) $ is a classical solution of 
\begin{align}\label{v_2}
\begin{cases}
y'(t)=\widetilde{\mathcal L} y(t)-P^-[n_1'(t)-\widetilde n_1'(0)],\quad t\geq 0\,,\\
y(0)= 0\,, \\
B_j y(t)=0,\quad j=1,\dots,m,\quad t\geq 0\,. 
\end{cases}
\end{align}
Moreover, since    $\widetilde{\mathcal L } \widetilde n_1(t),n_{1}'(t) \in B([1, \infty); C^{\a}(\o)) $, we obtain similarly that $v_2$ satisfies the same properties as $v_1$ (see \eqref{property}).

Let us consider $v_3$. We set
for each 
 \begin{equation*}
 \begin{cases}
 s=1,\dots , m & \text{ if } m_1=0\,,\\
 s=2,\dots, m  & \text{ if } m_1>0\,, 
 \end{cases}
 \end{equation*}
\begin{equation*}
\psi_s(t)=P^-\mathcal N_s \mathcal M_s (\widetilde g_1(t)- \widetilde g_1(0),\dots,  \widetilde g_s(t)- \widetilde g_s(0)),\quad t\in [0,\infty)\,,
\end{equation*}
and
$$
v_{3s}(t) = \int_0^t \! e^{(t-s')\widetilde{L}}
           P^- \psi_s(s')\, \mathrm{d}s' \,.
$$  
Therefore 
\begin{equation} \label{v_3 sum}
v_3(t)= P^- n_2(0)-\widetilde L\sum\limits_{s=1 \text{ or }2}^m v_{3s}(t) \,.
\end{equation}

We have 
\begin{equation}\label{Psi_s}
 \psi_s \in C^{\frac{2m+\a-m_s}{2m}}([0,\infty);\,\, D_{\widetilde L}(\tfrac{m_s}{2m}, \infty))
 \end{equation}
  because of the fact that $B_j \mathcal N_s =0$ for $j<s$. See (32) of $\S 7 $ in \cite{lunardisinestrariwah} for more details. Applying \cite[Theorem~4.3.16]{Lunardi1995424} with $\theta =\frac{2m+\a-m_s}{2m} $, $\beta =\frac{m_s}{2m} $, we obtain for every T>0
\begin{align*}
\widetilde L v_{3s}\in C^{1+\frac{\a}{2m}}([0,T];X),\qquad v_{3s}' \in B([0,T];D_{\widetilde L}(1+\tfrac{\a}{2m}, \infty)).\vspace{10pt}
\end{align*}
By looking at the proof of Theorem 4.3.16  and Theorem 4.3.1(iii) in \cite{Lunardi1995424},   we see that 
$$
\|\widetilde L v_{3s}\|_{C^{1+\frac{\a}{2m}}([0,T];X)}+\|\widetilde L v_{3s}'\|_{B([0,T];D_{\widetilde L}(\frac{\a}{2m}, \infty))} \leq C \| \mathcal \psi_s \|_{C^{\frac{2m+\a-m_s}{2m}}([0,\infty);\,\, C^{m_s}(\o))}\,,
$$
with  the constant $C$ independent of $T$ and hence by \eqref{v_3 sum}
we get
$$ 
v_3' \in B([1, \infty); C^\a(\o))\cap C^\frac{\a}{2m}([1,\infty);X)
$$
and $v_3(t)=P^- \widetilde n_2(0)-\widetilde L z(t)$, where $z(t)$ is a classical solution of 
\begin{align}\label{v_3}
\begin{cases}
z'(t)=\widetilde{L} z(t)+P^-[\widetilde n_2(t)-\widetilde n_2(0)],\quad t\geq 0\,,\\
z(0)= 0\,. 
\end{cases}
\end{align}
Moreover by \eqref{v_3} we easily check that $$v_3'=\widetilde {\mathcal L} v_3-\widetilde {\mathcal L}P^- n_2 $$
and therefore $\widetilde {\mathcal L} v_3 \in B([1, \infty); C^\a(\o))$. Summing up we obtain  that $v_3$ satisfies the same properties as $v_1$ (see \eqref{property}).

We now consider $v_4$. Since again $t \rightarrow \widetilde f (\cdot, t)$, $t \rightarrow \widetilde{\mathcal L } \widetilde n(s)$  belong to 
$C^{\frac{\a}{2m}}([0,\infty);X)$, 
by \cite[Proposition 4.4.2(ii)]{Lunardi1995424} we  obtain  that  it satisfies the same properties as $v_1$ (see \eqref{property}).

Finally we consider $v_5$. Due to  the estimates \eqref{projection estimates}, $v_5$ is clearly bounded with values in $D(L^k)$ for every $k \in \mathbb{N} $. Moreover, Because  $L(I-P^{-})\widetilde n \in B([1, \infty);C^{\a}(\o))$,  $v_5'=\widetilde L v_5 - \widetilde L (I-P^-)\widetilde n$ is Hölder continuous  with exponent $ \frac{\a}{2m}$ with value in $X$ and is bounded with value in $C^{\a}(\o)$.  Hence $v_5$ satisfies  the same properties as $v_1$ (see \eqref{property}).

Since $v=\sum_{i=1}^5 v_i$, we have
\begin{align}\label{good estimate}
\begin{cases}
v\in C^{1+\frac{\a}{2m}}([1,\infty);X),\qquad v' \in B([1,\infty);C^{\a}(\o)),\vspace{10pt}\\
\widetilde{\mathcal L} v_{} \in B([1,\infty);C^\a(\o)), \qquad v(t) \in \bigcap\limits_{p>1} W^{2m,p}(\Omega), \quad t \in [1,\infty)\,.
\end{cases}
\end{align}
Now what is left is  to prove that 
$$
v \in B([1,\infty); \, C^{2m+\a}(\o))
$$
and this can be done by using (iv) of  Theorem \ref{known result}, by means of  \eqref{good estimate} and the fact that $B_j v=\widetilde g_j \in  B([1,\infty);C^{2m+\a-m_j}(\partial \Omega))$.
\vspace{6 pt}

It follows that $v \in C^{2m+\a, 1+\frac{\a}{2m}}(\o \times [1,\infty)),  $ and 
 $$
 \| v \|_{C^{2m+\a, 1+\frac{\a}{2m}}(\o \times [1,\infty))} \leq C( \|u_0\|_X + \|\widetilde f \|_{\mE_0(\infty)}+\| \widetilde g \|_{\mF(\infty)})\,,
 $$
which finishes the proof.
\end{proof}  

\subsection{An extension operator}\label{Sec. extension operator}
 In order to apply the semigroup theory, similarly as in the previous section,  to  obtain results for the asymptotic behavior of linear systems (see next section), we need to  construct explicitly an extension operator for the case of vector-valued unknowns.

Let us recall our  linear boundary problem:
\begin{equation}\label{appendix BCs}
 (B_ju)(x)=\sum_{|\beta|\leq m_{j}}b^{j}_\beta(x)\nabla^{\beta}u (x)\,,\quad x\in \partial \Omega,\quad j=1,\dots,mN \,. 
\end{equation}
Here $ u : \o \times [0, \infty) \rightarrow \mR^N$, $b^j_\beta$ are $N$-dimensional row-vectors and 
\begin{equation*}
0 \leq m_1 \leq m_2 \leq \cdots \leq m_{mN} \leq 2m-1 \,.
\end{equation*}
Our goal is to construct explicitly  a linear and bounded operator $E$ such that for all $\theta' \in [0, \a] $,    
\begin{equation}\label{extension operator}
\left \{
 \begin{aligned}
    &g_j \in C^{2m+\theta'-m_j}(\partial \Omega), \quad j=1,\dots,mN \implies     E(g_1, \dots, g_{mN}) \in C^{2m+\theta'}(\o),\\[3pt]
    &B_jE(g_1,\dots,g_{mN})=g_j,\quad j=1,\dots,mN\,.
 \end{aligned}
\right.
\end{equation}
Note that the case $N = 1$ is treated in \cite[Theorem~6.3]{lunardisinestrariwah}, i.e., Theorem \ref{scalar result theorem}.

The strategy for proving the existence of the extension operator $E$ satisfying \eqref{extension operator} is as follows: At first, by using the  normality condition \eqref{normality condition}, we will reduce our linear system to an uncoupled linear system and then with the help of the scalar result, i.e., Theorem \ref{scalar result theorem}, applying it to each component, we finish the proof. 

In the following, we set $\gamma_j $ for the $j^{\mbox{th}}$-order normal derivatives precisely, for $j=0,\dots, 2m-1$
\begin{equation*} 
\gamma_j u :=D^j u\overbrace{[\nu,\cdots,\nu]}^{j-\text{times}} |_{\partial \Omega}\,,
\end{equation*}
which should be understood component-wise. Remind that $\nu(x)$ is the unit outer normal to $\partial \Omega $ at the point $x$ and $n_k \geq 0$ are the number of $k^{\mbox{th}}$-order boundary conditions for $k=0,\dots, 2m-1$.

\begin{theo}\label{theorem extension}
 Assume the operators $B_j$ satisfy  the regularity condition (H2) and the normality condition \eqref{normality condition}. Then there exists a linear bounded operator $E$ satisfying \eqref{extension operator}. 
\end{theo}
\begin{proof}
 Without loss of generality we  assume that $n_k \neq 0$ for all $k$ between  $0$ and $2m-1,$ i.e., we have here included all orders  $k$ between $0$ and $2m-1$. Indeed, if $n_k=0 $ for some $k,$ we could simply add the boundary conditions $\gamma_k u=0$.

Let $E$ be defined by 
\begin{align}\label{E}
E(g_1,&\dots, g_{mN}) \\ \nonumber
&:=\Big(\mathcal{N}(\psi_{01},\psi_{11},\dots,\psi_{2m-1,1}), 
\dots,\mathcal{N}(\psi_{0N},\psi_{1N},\dots,\psi_{2m-1,N})\Big),
\end{align}
 where the operator 
 $$ 
 \mathcal{N}(\psi_{0i},\psi_{1i},\dots,\psi_{2m-1,i})=\sum_{s=1}^{2m} \mathcal{N}_s \mathcal{M}_s(\psi_{0i},\dots,\psi_{s-1,i})\,
 $$
 is the extension operator given in Theorem \ref{scalar result theorem} for the boundary operators $B_j=\gamma_{j-1}  $, for $j=1,\dots, 2m$.
More precisely, $u=E(g_1,\dots,g_{mN})$ solves the following uncoupled linear system of normal boundary conditions:
\begin{align}\label{uncoupled}
\begin{cases}
\gamma_0 u=\psi_0 \,, \\
\gamma_1 u=\psi_1 \,,\\
\hspace{1cm}\vdots\\
\gamma_{2m-1} u =\psi_{2m-1} \, ,
\end{cases}
\end{align} 
where
$$
 \psi_k(x)=
  \begin{pmatrix}
     \psi_{k1}(x)\\
     \vdots\\
     \psi_{kN}(x)
  \end{pmatrix}
$$
will be defined below.
Note that by looking at the proof of Theorem \ref{scalar result theorem} or equivalently Theorem 6.3 in \cite{lunardisinestrariwah}, one sees that the number of boundary conditions in Theorem 6.3 in \cite{lunardisinestrariwah} can be replaced  by any $m'$ as far as the normality condition is satisfied and $m_j \leq 2m-1$ for all $j=1,\dots, m'$, which is definitely the   case in our situation. 

Setting $u=E(g_1,\dots,g_{mN})$ in \eqref{extension operator} and decomposing derivatives into  normal  and tangential derivatives,  the last condition in \eqref{extension operator} can be rewritten as
\begin{equation}\label{tangential and normal}
\sum_{i=0}^{j} S_{j,i} \gamma_i u = \varphi_j \,,
\end{equation}
where 
$S_{j,i}$ are tangential differential operator of order at most $j-i$ and
\begin{align*}
  \varphi_0:=
  \begin{pmatrix}
  g_1\\ 
  \vdots \\
  g_{n_0}
  \end{pmatrix}_{n_0 \times 1}\,,\qquad
  \varphi_{k+1}:=
  \begin{pmatrix}
   g_{\sum_{i=0}^{k} n_{i}+1}\\
   \vdots \\
    g_{\sum_{i=0}^{k+1} n_{i}}
   \end{pmatrix}_{n_{k+1}\times 1}\,. 
       \end{align*}
In particular for all $k=0,\dots,2m-1$
\begin{align*}
  S_{kk}(x)=
  \begin{pmatrix}
   \sum_{| \beta |= k}b^{j_1}_\beta(x)(\nu(x))^\beta  \\ 
   \vdots \\
   \sum_{| \beta | =k}b^{j_{n_{k}}}_\beta(x)(\nu(x))^\beta
  \end{pmatrix}_{n_k \times N}  \,,
  \end{align*}
  where $\{ j_{i} : i = 1, \dots, n_k \} = \{ j : m_j = k \} $ and for $j = 0, 1$ in \eqref{tangential and normal} we have  
\begin{align} 
\begin{cases}
S_{00}(x)\gamma_0 u=\varphi_0 \,, \\
 S_{11} (x) \gamma_1 u + \text{ tangential derivatives}+\text{zeroth order normal derivatives} =\varphi_1 \,.
\end{cases}
\end{align}
By the normality condition, $S_{kk}$ are surjective  and therefore  there exist  matrices $R_{kk}$ which have the same regularity as $S_{kk}$ such that 
\begin{align}\label{right inverse}
S_{kk} R_{kk}=I \quad \text{ on } \mR^{n_k}\,.
\end{align}

Now we are in a position to define $\psi_k$ such that \eqref{tangential and normal} holds. 
Define $\psi_0 := R_{00} \varphi_0$. Then  
\begin{equation}
S_{00} \gamma_0 u = S_{00} \psi_0 = S_{00}R_{00} \varphi_0 = \varphi_0 \,,
\end{equation}
that is,   \eqref{tangential and normal} is satisfied for $j =0$.  
Let us now consider  $j = 1$ which corresponds to the first-order boundary conditions.  Using the fact that $ \gamma _0 u =\psi_0 = R_{00} \varphi_0$ all tangential derivatives and of course all zeroth-order normal derivatives can be calculated in terms of  $R_{00} \varphi_0$.  Consequently the condition \eqref{tangential and normal} for $j= 1$ can be rewritten as
\begin{align*}
S_{11} (x) \gamma_1 u=\varphi_1(x) + \eta_1(x)\,,
\end{align*}
for some $\eta_1(x)$ which can be calculated in terms of  $R_{00} \varphi_0 $ or precisely in terms of $(g_1,\dots ,g_{n_0})$.
Therefore, by  defining $ \psi_1 := R_{11}(\varphi_1 + \eta_1) $  we are done with the case $j=1$.  By iteration, we define
$$
\psi_k := R_{kk}(\varphi_k + \eta_k)
$$
for some $\eta_k$ which can be calculated in terms of $\psi_0,\dots,\psi_{k-1}$.  
Moreover, by \eqref{N_s} for each $v \in C(\partial \Omega)$
\begin{align}
(B_j(\mathcal N_s v_1, \dots,\mathcal N_s v_N))(x)\equiv 0\,, \quad x\in \partial \Omega \,,\quad \text{ } m_j<s-1 \,.
\end{align}
And finally the regularity condition in \eqref{extension operator} comes from the fact that the operator $\mathcal N$ has a similar regularity property, see \eqref{mathcal N1}, and this finishes the proof. 
 \end{proof}

\subsection{Asymptotic behavior in linear systems}
Here we  extend the result of Section \ref{abilp} to the systems of $mN$ boundary conditions  for a linear system.  Precisely we consider the  linear problem \eqref{linear}, 
 i.e., 
\begin{equation}
\left\{\begin{array}{lll}
\partial _t u+Au=f(t)&\mbox{in }\Omega,&t \geq0, \\[0.1cm]
Bu=g(t)&\mbox{on }\partial\Omega,&t \geq 0, \\[0.1cm]
u(0)=u_{0}&\mbox{in }\Omega, \\[0.1cm]
\end{array}\right.
\label{linear system1}
\end{equation}
where
 $ u: \o \times [0,\infty) \rightarrow \mR^N$, $\Omega$ is a bounded domain in $\mathbb{R}^n$ with $C^{2m+\a}$ boundary, $0<\a<1$, $g=(g_1,\dots,g_{mN}), B=(B_1, \dots, B_{mN})$, $u_0 \in C^{2m+\a}(\o)$ and the operators $A$ and $B$ satisfy the conditions (H2), (L-S), (SP) and  the normality condition  \eqref{normality condition}. 

 Theorem \ref{known result} (i) states that the realisation $-A_0$ of $-A$ with homogeneous boundary conditions in $C(\o)$, defined in \eqref{linear operator}, is a sectorial operator.

  Furthermore if $f \in \mathbb E_0(T)$, $g \in \mathbb F(T)$ and $u_0 \in C^{2m+ \a}(\o)$ satisfying the compatibility condition \eqref{compatibility linear}, the unique solution of \eqref{linear system1} belongs to $\mE_1(T)$ for all $T$ and in addition it is given by the extension of the Balakrishnan formula with some adaptations. Indeed, by our explicit construction of the extension operator (see \eqref{E}), we simply can extend  Theorem 4.1 in \cite{lunardisinestrariwah}
to cover  the linear systems (using the same technique). Therefore the following representation formula holds for each $t \in [0,T]$:

\begin{align}\label{Balakrishnan formula systems}
u(\cdot,t)&= e^{tL}(u_0-n(0))+\int_0^t \! e^{(t-s)L}[f(\cdot,s)+\mathcal{L}
 n(s)-n_1'(0)]\mathrm{d}s\nonumber\\
 &\quad +n_1(t)-\int_0^t \! e^{(t-s)L}(n_1'(s)-n_1'(0))\,\mathrm{d} s \nonumber\\
  & \quad-L \int_0^t \!e^{(t-s)L}[n_{2}(s)- n_2(0)] \, \mathrm{d} s +n_2(0)\nonumber\\
&= e^{tL} u_0 + \int_0^t \! e^{(t-s)L}[f(\cdot,s)+\mathcal{L}n(s)]
\, \mathrm{d} s  \nonumber\\
&\quad -L\int_0^t \! e^{(t-s)L}n(s)\,\mathrm d s\,, \quad    
\end{align}
with  $\mathcal{L}=-A$ and $ L=-A_0$. Here 
$$
n(t)= E( g_1(t), \dots, g_{mN}(t))
$$
and similarly as before 
\begin{equation*}
n_1(t)=\begin{cases}
         0& \text{ if } m_1>0 \,,\\
         \left(\mathcal N_1 \mathcal M_1 ( \psi_{0,1}), \dots, \mathcal N_1 \mathcal M_1 (\psi_{0,N})\right) & \text{ if } m_1=0\,
       \end{cases}
\end{equation*}
and $ n_2(t)= n(t)- n_1(t)$,
 where $\psi_0=(\psi_{0,1}, \dots, \psi_{0,N})^T=R_{00} \varphi_0$ which can be written in terms of $g_1, \dots, g_{n_0}$.  

\begin{theo}\label{asymptotic behavior theorem for systems}
Let $0< \omega <- \mathrm{max}\,\{ \mathrm{Re}\, \lambda :\lambda \in \sigma^-(-A_0) \}$. Suppose $f$ and $g$  are such that $(\sigma,t)\rightarrow\ e^{\omega t}f(\sigma, t)\in \mE_0(\infty)$ and $(\sigma , t)\rightarrow e^{\omega t}g(\sigma, t) \in \mF(\infty)$. Suppose further that $u_0 \in C^{2m+\a}(\o)$ satisfy the compatibility condition \eqref{compatibility linear}. Let  $u$ be  the solution of \eqref{linear system1}.
Then $v(\sigma , t)= e^{\omega t}u(\sigma, t)$ is bounded in $[0,+\infty)\times \o $ if and only if 
\begin{align}\label{condition on u_0 system}
(I-P^-)u_0=&-\int_0^{+\infty} \! e^{-sL}(I-P^-)[f(\cdot,s)+\mathcal{L}E
g(\cdot,s)]\,\mathrm{d}s \nonumber\\
           &+L\int_0^ \infty \! e^{-sL} (I-P^-)Eg(\cdot, s)\, \mathrm{d}s\,. \end{align}
In this case, the function $u$ is given by 
\begin{align}
u(\cdot,t)  &= e^{tL} P^- u_0+ \int_0^t \! e^ {(t-s)L} P^-[f(\cdot, s)+\mathcal{L}Eg(\cdot,s)] \, \mathrm{d} s \nonumber \\ 
            &\quad -L \int_0^t \! e^{(t-s)L}P^-Eg(\cdot,s)\, \mathrm{d} s \nonumber\\ 
            &\quad -\int_t^{+ \infty }\! e^{(t-s)L}(I-P^-)[f(\cdot,s)+\mathcal{L} E g(\cdot,s)]\, \mathrm{d}s  \nonumber\\
            &\quad +L \int_t^{+\infty} \! e^{(t-s)L}(I-P^-)Eg(\cdot,s)\,                   \mathrm{d} s,\label{expression for u in systems}        
\end{align}
and the function $v=e^{\omega t}u$ belongs to $\mE_1(\infty)$, with the estimate
\begin{align*}
\|& v \|_{\mE_1(\infty)} \leq C( \| u_0 \|_{C^{2m+\a }(\o)} + \| e^{ \omega t} f \|_{\mE_0(\infty)}+\|e^{ \omega t}g\|_{\mF(\infty)}).  
\end{align*}
\end{theo}
\begin{proof}
The proof is exactly the same as the one of Theorem \ref{asymptotic behavior theorem}. More precisely, as you have seen, we  used the abstract theories in the proof, i.e., the theory of semigroups of linear operators, except for the part related to the function $v_3$. Due to   our explicit construction of the extension operator (see \eqref{E}) and taking into account \eqref{N_s} (in order to obtain the same result as   \eqref{Psi_s}), we can work component-wise and get the same estimate for the function $v_3$. This finishes the proof.
\end{proof}
In the stable case, i.e, when $\sigma (-A_0) = \sigma^-(-A_0)$, We immediately get the following corollary of Theorem \ref{asymptotic behavior theorem for systems}.
\begin{cor}\label{corollary asymptotic}
Let $\omega_A:=\sup\{Re\lambda :\lambda \in \sigma(-A_0)\}<0$ and  $\omega \in (0, -\omega_A)$. Assume $f$ and $g$ are such that $(\sigma,t)\rightarrow\ e^{\omega t}f(\sigma, t)\in \mE_0(\infty)$ and $(\sigma , t)\rightarrow e^{\omega t}g(\sigma, t) \in \mF(\infty)$ and let $u_0 \in C^{2m+\a}(\o)$ satisfy the compatibility condition \eqref{compatibility linear}.
Let  $u$ be  the solution of \eqref{linear system1}, where $u \in \mE_1(T) $ for all $T< \infty$. Then $v(\sigma , t)= e^{\omega t}u(\sigma, t)$ belongs to $\mE_1(\infty)$ and  
\begin{align*}
\|& v \|_{\mE_1(\infty)} \leq C( \| u_0 \|_{C^{2m+\a }(\o)} + \| e^{ \omega t} f \|_{\mE_0(\infty)}+\|e^{ \omega t}g\|_{\mF(\infty)}).  
\end{align*}
\end{cor}

\subsection{Proof of  Proposition \ref{nonlinear estimate}}\label{proof of nonlinear estimates}
In fact we are following the steps in the proof of  Theorem 4.1 in \cite{Lunardi2002385}.\\
For $0\leq t \leq a $, 
\begin{align*}
\|e^{\sigma t}F\l z_1\l t,\cdot\r\r-e^{\sigma t}F(z_2(t,\cdot))\|_{C^\a(\o)}&\leq K(r)\norm{e^{\sigma t}(z_1(t,\cdot)-z_2(t,\cdot))}_{C^{2m+\a}(\o)}\\
&\leq K(r)\norm{e^{\sigma t}(z_1-z_2)}_{\mE_1(a)},
\end{align*}
\begin{align*}
\norm{e^{\sigma t}G_j\l z_1\l t,\cdot\r\r-e^{\sigma t}G_j(z_2(t,\cdot))}_{C^{2m+\a-m_j}(\partial \Omega)}&\leq H_j(r)\norm{e^{\sigma t}(z_1(t,\cdot)-z_2(t,\cdot))}_{X_{1}}\\
&\leq H_{j}(r)\norm{e^{\sigma t}(z_1-z_2)}_{\mE_1(a)},
\end{align*}
while for $0\leq s\leq t \leq a$,
\begin{eqnarray*}
&& \hspace{-25pt}\|e^{\sigma t}F(z_1(t,\cdot))-e^{\sigma t}F(z_2(t,\cdot))-e^{\sigma               s}F(z_1(s,\cdot))+e^{\sigma s}F(z_2(s,\cdot))\|_{C(\o)}\\
&=& \left\| \rule{0cm}{15pt} \int_0^1 \! e^{\sigma t} F'\l\lambda z_1(t,\cdot)+(1-\lambda               )z_{2}(t,\cdot)\r(z_1(t,\cdot)-z_2(t,\cdot)) \right.\\ 
&& \left. \hspace{5pt}-e^{\sigma s} F'\l\lambda z_1(s,\cdot)+(1-\lambda )z_{2}(s,\cdot)\r(z_1(s,\cdot)-z_2(s,\cdot))\,\mathrm{d}\lambda                \rule{0cm}{15pt}\right \|_{C(\o)} \\
&\leq& \int_0^1\!\Big \| \left(\ F'\l \rule{0cm}{10 pt}\lambda z_1(t,\cdot)+(1-\lambda )z_{2}(t,\cdot)\r-              F'\l \rule {0cm}{10pt}\lambda z_1(s,\cdot)+(1-\lambda )z_{2}(s,\cdot)\r\right)\cdot\\
&& \hspace{20pt}\cdot  e^{\sigma t} (z_1(t,\cdot)-z_{2}(t,\cdot))\Big\|_{C(\o)}\,\mathrm{d}\lambda\\
&&  + \int_0^1 \left\|F'\l \rule {0cm}{10pt}\lambda z_1(s,\cdot)+(1-\lambda )z_{2}(s,\cdot)\r \cdot\right.\\
&&  \hspace{32pt}\cdot\l \left. \rule{0cm}{10pt} e^{\sigma t}(z_1(t,\cdot)-z_{2}(t,\cdot))-e^{\sigma                 s}(z_1(s,\cdot)-z_{2}(s,\cdot))\r \right\|\, \mathrm{d}\lambda \\
&\leq& \frac{L}{2} \l \|z_1(t,\cdot)-z_{1}(s,\cdot)\|_{C^{2m}(\o)}+\|z_2(t,\cdot)-z_{2}(s,\cdot)\|_{C^{2m}(\o)}\r               e^{\sigma t} \|z_1(t,\cdot)-z_{2}(t,\cdot)\|_{C^{2m}(\o)}\\
&& + Lr \|e^{\sigma t}\l z_1(t,\cdot)-z_2(t,\cdot)\r-e^{\sigma s }(z_1(s,\cdot)-z_2(s,\cdot))\|_{C^{2m}(\o)}\\
&\leq& \frac{L}{2}(t-s)^{\frac{\a}{2m}}(\|z_1\|_{C^{\frac{\a}{2m}}((0,a),C^{2m}(\o))}+\|z_2\|_{C^{\frac{\a}{2m}}((0,a),C^{2m}(\o))})\|e^{\sigma               t}(z_1 -z_2)\|_{\mE_1(a)} \\
&& + Lr(t-s)^{\frac{\a}{2m}} \|e^{\sigma t}\l z_1-z_2\r\|_{C^{\frac{\a}{2m}}((0,a),C^{2m}(\o))}\\
&\leq&    2Lr (t-s)^{\frac{\a}{2m}}\|e^{\sigma t}(z_1-z_2)\|_{\mE_1(a)}\,.
\end{eqnarray*}
The last inequality is a consequence of Lemma \ref{parabolic estimate} and the fact that $z_1, z_2\in B_{\mathbb{E}_1(a)}(0,r)$.
\vspace{5pt}

Since $1+\frac{\a}{2m}-\frac{m_j}{2m}<1$  for $j$ with  $m_j\ge 1$,    we get similarly \begin{eqnarray*}
&&\|e^{\sigma t}G_j(z_1(t,\cdot))-e^{\sigma t}G_j(z_2(t,\cdot))-e^{\sigma s}G_j(z_1(s,\cdot))+e^{\sigma s}G_j(z_2(s,\cdot))\|_{C(\partial \Omega)}\\
&&\quad\le2Lr(t-s)^{1+\frac{\a}{2m}-\frac{m_j}{2m}}\|e^{\sigma t}(z_1-z_2)\|_{\mE_1(a),}
\end{eqnarray*}
 where we have used the  embedding \begin{eqnarray*}
\mE_1(a)\hookrightarrow C^{1+\frac{\a}{2m}-\frac{m_j}{2m}}((0,a),C^{m_{j}}(\o)),
\end{eqnarray*}
which is a consequence of Lemma \ref{parabolic estimate}.

For $j$ such that $m_j=0$, we have to estimate the complete norm, i.e.,  
$$
\|e^{\sigma t }(G_j(z_1)-G_j(z_2))\|_{C^{1+\frac{\a}{2m}}(I,C(\partial \Omega))}\,,
$$ 
which includes the time derivative. The proof is again similar, but for the convenience we give some details of the main part of it namely, estimating $$\| e^{\sigma t} \frac{d}{dt}(G_j(z_1)-G_j(z_2))\|_{C^{\frac{\a}{2m}}(I,C(\partial  \Omega))} \,.$$
Note that exactly at this point one needs $C^2$-regularity for  $G_j$.

For $0\leq s\leq t \leq a$, we have
\begin{eqnarray*}
&& \hspace{-25pt} \left\|e^{\sigma t}G_j'(z_{_{1}}(t,\cdot))z_{1}'(t,\cdot)-e^{\sigma t}G_j'(z_2(t,\cdot))z_2'(t,\cdot)-\right.\\ 
&&\left.\hspace{3cm}-e^{\sigma s}G_j'(z_{_{1}}(s,\cdot))z_{1}'(s,\cdot)
+e^{\sigma s}G_j'(z_2(s,\cdot))z_2'(s,\cdot)\right \|_{C(\partial \Omega)} \\
&\leq&  \left\| \rule{0cm}{15pt}\int_0^1 \!e^{\sigma t}G_j''\l \rule{0cm}{10 pt}\lambda z_1(t,\cdot)+(1-\lambda)z_2(t,\cdot)\r\l \rule{0cm}{10pt}z_1(t,\cdot)-z_2(t,\cdot)\r \cdot \right.\\
&&\left. \hspace{0.36cm}\cdot \l\rule{0cm}{10pt}\lambda z_1'(t,\cdot)+(1-\lambda) z_2'(t,\cdot)\r -e^{\sigma s}G_j''\l \rule{0cm}{10 pt}\lambda z_1(s,\cdot)+(1-\lambda)z_2(s,\cdot)\r \cdot\right.\\
&& \left.\hspace{2.7cm} \cdot\l \rule{0cm}{10pt}z_1(s,\cdot)-z_2(s,\cdot)\r\l\rule{0cm}{10pt}\lambda z_1'(s,\cdot)+(1-\lambda)
 z_2'(s,\cdot)\r \,\mathrm{d}\lambda\rule{0cm}{15pt}\right\|_{C(\partial \Omega)}\\
&&     +\left\|\int_0^1\!e^{\sigma t}G_j'\l \rule{0cm}{10 pt}\lambda z_1(t,\cdot)+(1-\lambda)z_2(t,\cdot)\r\l \rule{0cm}{10pt}z_1'(t,\cdot)-z_2'(t,\cdot)\r\right.\\
&&     \left.\hspace{20pt}-e^{\sigma s}G_j'\l \rule{0cm}{10 pt}\lambda z_1(s,\cdot)+(1-\lambda)z_2(s,\cdot)\r\l \rule{0cm}{10pt}z_1'(s,\cdot)-z_2'(s,\cdot)\r\,\mathrm{d}\sigma\right\|_{C(\partial \Omega)}\\
\end{eqnarray*}
\vspace{-1.1cm}
\begin{eqnarray*}
&\leq& \int_0^1\!\left\| \l G_j''\l \rule{0cm}{10 pt}\lambda z_1(t,\cdot)+(1-\lambda)z_2(t,\cdot)\r-G_j''\l \rule{0cm}{10 pt}\lambda z_1(s,\cdot)+(1-\lambda)z_2(s,\cdot)\r\r \cdot\right.\\
&&\left.\qquad\cdot e^{\sigma t}\l \rule{0cm}{10pt}z_1(t,\cdot)-z_2(t,\cdot)\r\l\rule{0cm}{10pt}\lambda z_1'(t,\cdot)+(1-\lambda) z_2'(t,\cdot)\r \right\|_{C(\partial \Omega)}\,\mathrm{d}\lambda\\
&&+\int_0^1 \! \left\|G_j''\l\rule{0cm}{10pt}\lambda z_1(s,\cdot)+(1-\lambda) z_2(s,\cdot)\r \l  e^{\sigma t}\l \rule{0cm}{10pt}z_1(t,\cdot)-z_2(t,\cdot)\r-\right.\right.\\
&&\left.\left.\hspace{32pt} -e^{\sigma s}\l \rule{0cm}{10pt}z_1(s,\cdot)-z_2(s,\cdot)\r\r\times \l\rule{0cm}{10pt}\lambda z_1'(t,\cdot)+(1-\lambda) z_2'(t,\cdot)\r\right\|_{C(\partial \Omega)}\,\mathrm{d}\lambda
\end{eqnarray*}
\vspace{-0.65cm}
\begin{eqnarray*}
&& +\int_0^1\! \left\|G_j''\l\rule{0cm}{10pt}\lambda z_1(s,\cdot)+(1-\lambda) z_2(s,\cdot)\r  e^{\sigma s}\l \rule{0cm}{10pt}z_1(s,\cdot)-z_2(s,\cdot)\r\cdot\right.\\
&&\left.\hspace{32pt}\cdot \l\rule{0cm}{10pt}\lambda z_1'(t,\cdot)+(1-\lambda) z_2'(t,\cdot)-\lambda z_1'(s,\cdot)-(1-\lambda)z_2'(s,\cdot)\r\right\|_{C(\partial \Omega)}
\,\mathrm{d}\lambda
\end{eqnarray*}
\vspace{-0.55cm}
\begin{eqnarray*}
&& + \int_0^1 \! \left\| \l G_j' \l \rule{0cm}{10 pt} \lambda z_1(t,\cdot)+(1- \lambda) z_2(t,\cdot)\r-G_j'\l \rule{0cm}{10 pt}\lambda z_1(s,\cdot)+(1-\lambda)z_2(s,\cdot)\r \r \cdot\right.
\end{eqnarray*}
\vspace{-0.5cm}
\begin{eqnarray*}
&&\hspace{32pt}\left. \cdot e^{\sigma t} \l \rule{0cm}{10 pt}z_1'(t,\cdot)-z_2'(t,\cdot)\r \right\|_{C(\partial \Omega)}\,\mathrm{d} \lambda \\
&& + \int_0^1 \! \left\| G_j'\l \rule{0cm}{10pt} \lambda z_1(s,\cdot)+(1-\lambda)z_2(s,\cdot) \r \cdot \right. \\&& \hspace{1.75cm}\cdot \left.\l \rule{0cm}{10pt} e^{\sigma t} \l \rule{0cm}{10pt}z_1'(t,\cdot)-z_2'(t,\cdot) \r - e^{\sigma s} \l \rule{0cm}{10pt}z_1'(s,\cdot)-z_2'(s,\cdot) \r \r \right\|_{C(\partial \Omega)} \, \mathrm{d} \lambda \\ 
&\leq& \left(12Lr^2 +2Lr  \right)(t-s)^{\frac{\a}{2m}}\|e^{\sigma t}(z_1-z_2)\|_{\mE_1(a)}\,, \end{eqnarray*}
where we have used the fact that $\lambda,1-\lambda \leq 1$ and 
$$
\mE_1(a)\hookrightarrow C^{\frac{\a}{2m}}(I, C^{2m}(\o))\hookrightarrow C^{\frac{\a}{2m}}(I, C(\o)) \,.
$$
Summing up, by adding all constants  we get a constant $D(r)$ such that $D(r)$ goes to zero as $r \longrightarrow 0$    and so the statement follows.